\documentclass[11pt]{amsart}
\usepackage{amsmath, amssymb, amsthm, amsfonts, mathrsfs}
\usepackage[all,cmtip]{xy}

\newtheorem{theorem}{Theorem}[section]
\newtheorem{proposition}[theorem]{Proposition}
\newtheorem{corollary}[theorem]{Corollary}
\newtheorem{lemma}[theorem]{Lemma}

\theoremstyle{definition}
\newtheorem{definition}[theorem]{Definition}

\theoremstyle{definition}

\theoremstyle{definition}

\DeclareSymbolFont{AMSb}{U}{msb}{m}{n}
\DeclareMathSymbol{\N}{\mathbin}{AMSb}{"4E}
\DeclareMathSymbol{\Z}{\mathbin}{AMSb}{"5A}
\DeclareMathSymbol{\R}{\mathbin}{AMSb}{"52}
\DeclareMathSymbol{\Q}{\mathbin}{AMSb}{"51}
\DeclareMathSymbol{\I}{\mathbin}{AMSb}{"49}
\DeclareMathSymbol{\C}{\mathbin}{AMSb}{"43}

\begin{document}
\title[Topological isomorphism for rank-1 systems]{Topological isomorphism for rank-1 systems}
\author{Su Gao}
\address{Department of Mathematics\\ University of North Texas\\ 1155 Union Circle \#311430\\  Denton, TX 76203\\ USA}
\email{sgao@unt.edu}
\author{Aaron Hill}
\address{Department of Mathematics\\ University of North Texas\\ 1155 Union Circle \#311430\\  Denton, TX 76203\\ USA}
\email{Aaron.Hill@unt.edu}
\date{\today}
\thanks{The first author acknowledges the US NSF grants DMS-0901853 and DMS-1201290 for the support of his research. The second author acknowledges the US NSF grant DMS-0943870 for the support of his research.}
\maketitle \thispagestyle{empty}

\begin{abstract}
We define the Polish space $\mathcal{R}$ of non-degenerate rank-1 systems.  Each non-degenerate rank-1 system can be viewed as a measure-preserving transformation of an atomless, $\sigma$-finite measure space and as a homeomorphism of a Cantor space.  We completely characterize when two non-degenerate rank-1 systems are topologically isomorphic.  We also analyze the complexity of the topological isomorphism relation on $\mathcal{R}$, showing that it is $F_{\sigma}$ as a subset of $\mathcal{R} \times \mathcal{R}$ and bi-reducible to $E_0$.  We also explicitly describe when a non-degenerate rank-1 system is topologically isomorphic to its inverse.
\end{abstract}

\tableofcontents


\section{Introduction}

\subsection{Two open questions in ergodic theory}

One motivation of the work in this paper comes from two open questions concerning the conjugacy action on the group of invertible measure-preserving transformations.  Let Aut$(X, \mu)$ denote the group of invertible, measure-preserving transformations of a standard Lebesgue space, taken modulo null sets and equipped with the weak topology.  With the topology comes a notion of genericity.  A subset of Aut$(X, \mu)$ is {\em generic} if its complement is a countable union of nowhere dense sets.  We then say that a generic transformation satisfies a certain property if that property defines a generic subset of Aut$(X, \mu)$.

The group Aut$(X, \mu)$ acts on itself by conjugation, and the two open questions come from complementary aspects of this action.  The first question deals with the orbit structure of the action:  How can one determine whether two elements of Aut$(X, \mu)$ are conjugate (i.e., in the same orbit)?  As conjugacy is the natural choice for a notion of isomorphism in this context, this question is frequently referred to as the isomorphism problem for invertible measure-preserving transformations.  

The second question deals with stabilizers.  An informal version of this question is this:  If one chooses an element of Aut$(X, \mu)$ at random, what properties will the stabilizer of that transformation likely have?  It is clear that the stabilizer of an element of Aut$(X, \mu)$ under the conjugacy action is exactly the centralizer of that element in the group Aut$(X, \mu)$.  Thus, this problem is sometimes called the centralizer problem for invertible measure-preserving transformations.  

\subsubsection{The isomorphism problem}

 The isomorphism problem was first formulated by von Neumann, who together with Halmos gave a very important partial solution (see \cite{HalmosvonNeumann}).  They showed that for ergodic transformations with pure point spectrum, the spectrum is a complete invariant.  That is, two ergodic transformations with pure point spectrum are conjugate if and only if they have the same spectrum.  Another very important partial solution was given by Ornstein, who showed in \cite{Ornstein} that for Bernoulli shifts, entropy is a complete invariant.
 
In the past few decades, however, there have been several papers showing that certain types of ``nice" classifications are impossible for all of Aut$(X, \mu)$ (or even for a generic subset of Aut$(X, \mu)$).  One type of ``nice" classification is connected with Borel reducibility and the strongest of the anti-classification results of this type was given by Foreman and Weiss in 2004.  Building on the work of Hjorth (\cite {Hjorth1}), they showed in \cite {ForemanWeiss} that no generic subset of Aut$(X, \mu)$ is classifiable by countable structures.

Another type of ``nice" classification deals with a description of the conjugacy (isomorphism) relation on Aut$(X, \mu)$ as a subset of $\textnormal{Aut}(X, \mu) \times \textnormal{Aut}(X, \mu)$.  Hjorth showed in \cite{Hjorth1} that this is non-Borel and Foreman, Rudolph and Weiss showed in \cite{ForemanRudolphWeiss} that when restricted to the (generic) class of ergodic transformations, the isomorphism relation is still not Borel.

However, in the same paper Foreman, Rudolph, and Weiss show that for a different generic class of transformations, the rank-1 transformations, the isomorphism relation is Borel.  Unfortunately, their proof is abstract and it gives no explicit way of determining when two rank-1 transformations are isomorphic and no bound on the Borel complexity of the isomorphism relation on rank-1 transformations.

It would be very nice have an explicit method of determining when two rank-1 transformations are (measure-theoretically) isomorphic.

\subsubsection{The centralizer problem}

We are interested in the size and structure of the centralizer of a generic transformation.  For any $T \in \textnormal{Aut}(X, \mu)$ the centralizer $C(T)$ contains $\overline{\{T^i : i \in \mathbb{Z}\}}$, since $\textnormal{Aut}(X, \mu)$ is a topological group.  An important result of Jonathan King is that for rank-1 transformations, there is equality (see \cite{King1}).  Since a generic transformation is rank-1, this imlies that the centralizer of a generic transformation is abelian and, in some sense, as small as possible.

On the other hand, recent results of King \cite{King2}, de Sam Lazaro--de la Rue \cite{deSamLazarodelaRue}, Ageev \cite{Ageev}, Stepin--Eremenko \cite{StepinEremenko}, Tikhonov \cite{Tikhonov}, Melleray--Tsankov \cite{MellerayTsankov}, and Solecki \cite{Solecki} show that the centralizer of a generic transformation is large and structurally rich.  For example, Stepin and Eremenko showed that every compact abelian group embeds into the centralizer of a generic transformation.  Mellerey and Tsankov showed that the centralizer of a generic transformation is extremely amenable.  Solecki showed that every element of the centralizer of a generic transformation is in a one-parameter subgroup of that centralizer.

It would be very nice to know whether there exists a Polish group $G$ so that the centralizer of a generic transformation is isomorphic to $G$.  If there is such a $G$, of course it would be nice to know what that $G$ is.

Because rank-1 transformations are generic, questions about the centralizer of a generic transformation are the same as the corresponding questions for a generic rank-1 transformation.  Because of King's weak closure theorem, these questions are the same as the corresponding questions about $\overline{\{T^i : i \in \mathbb{Z}\} }$, for generic rank-1 $T$.

\subsubsection{Rank-1 systems}

There are various definitions for rank-1 transformation in the literature, and not all of them are equivalent.  The two most common definitions involve cutting and staching transformations and symbolic rank-1 systems.  Every cutting and stacking transformation, except those isomoprhic to odometers, can be realized as (i.e., is isomophic to) a symbolic rank-1 system.  For further information regarding the many definitions of rank-1 transformation and connections between them, see Ferenczi's survey article \cite{Ferenczi}.  


In the symbolic definition of rank-1, one constructs a collection of concrete rank-1 systems, each concrete rank-1 system can be viewed either as a measure-preserving transformation of a standard Lebesgue space or as a homeomorphism of a Cantor space.  One then defines a measure-preserving transformation to be rank-1 if it is measure-theoretically isomorphic to a concrete rank-1 system.  One could analogously define a homeomorphism to be rank-1 if it is topologically isomorphic to a concrete rank-1 system.

For the open questions described above, one would be interested in:  When are two concrete rank-1 systems measure-theoretically isomorphic?  What measure-theoretic self-isomorphisms exist for particular concrete rank-1 systems?

The topological analogues would be:  When are two concrete rank-1 systems topologically isomorphic?  What topological self-isomorphisms exist for particular concrete rank-1 systems?  In this paper we will completely answer these topological analogues.

It should be noted that the systems that we are concerned with in this paper (what we call non-degenerate rank-1 systems) are more general than the constructive symbolic rank-1 systems in Ferenczi's article.  Our definition omits a restriction that guarantees that the rank-1 system can be viewed as a measure-preserving transformation of a standard Lebesgue space.  Without that restriction, the rank-1 system can be viewed as a measure-preserving transformation of a measure (but not necessarily probability) space.  Since we are interested in viewing rank-1 systems only as homeomorphisms of Cantor space, this restriction is not present in our definition.

\subsection{Main results}
The non-degenerate rank-1 systems we consider are all Bernoulli subshifts, that is, closed subsets of $\{0,1\}^\Z$ that are invariant under the shift map $\sigma$. For each non-degenerate rank-1 system $(X, \sigma)$, there is a particular $V\in \{0,1\}^\N$ that is associated to $(X, \sigma)$.  It is of a certain form (it is a non-degenerate rank-1 word) and is such that
$$ X=\{x\in \{0,1\}^\Z\,:\, \mbox{every finite subword of $x$ is a subword of $V$}\}. $$
The correspondence between non-degenerate rank-1 words and non-degenerate rank-1 systems is one-to-one, and we will consider the collection $\mathcal{R}$ of all non-degenerate rank-1 words as a space of codes for all non-degenerate rank-1 systems. We will define a natural (Polish) topology on $\mathcal{R}$ and regard $\mathcal{R}$ with this topology as the space of all non-degenerate rank-1 systems.

Our first main theorem gives a complete characterization when two non-degenerate rank-1 systems are topologically isomorphic. We define a simple form of block code, called a replacement scheme, for non-degenerate rank-1 systems, and call an isomorphism coming from a replacement scheme {\em stable}. 

\begin{theorem} 
If $\phi$ is a topological isomorphism between two non-degenerate rank-1 systems then there is some $i \in \Z$ so that $\phi \circ \sigma^i$ is a stable isomorphism.
\end{theorem}

As an immediate corollary, we are able to characterize all self-isomorphisms of a non-degenerate rank-1 system.

\begin{theorem} 
If $(X, \sigma)$ is a non-degenerate rank-1 system and $\tau$ is an autohomeomorphism of $X$ that commutes with $\sigma$, then for some $i \in \Z$, $\tau = \sigma^i$.
\end{theorem}

This result was proved in a slightly more general context by the second author, in \cite{Hill}.

Any non-degenerate rank-1 word is generated by an infinite sequence of finite words. Our analysis also allows us to identify a canonical generating sequence for a non-degenerate rank-1 word. The significance of this canonical generating sequence is that, if two non-degenerate rank-1 systems are isomorphic, then their respective generating sequences will eventually match up in some specific way. We are therefore able to determine the exact complexity of the topological isomorphism relation on non-degenerate rank-1 systems in the Borel reducibility hierarchy. Recall that $E_0$ is the eventual agreement relation on  $\{0,1\}^\N$, that is, the equivalence relation on $\{0,1\}^\N$ defined by
$$ xE_0y\iff \exists n\ \forall m>n\ x(m)=y(m). $$

\begin{theorem}\label{rank1isom} The topological isomorphism relation for non-degenerate rank-1 systems is Borel bi-reducible with $E_0$. Also, the topological isomorphism relation is $F_\sigma$ as a subset of $\mathcal{R}\times\mathcal{R}$.
\end{theorem}

Finally we also provide a satisfactory solution to the topological version of the inverse problem for non-degenerate rank-1 systems.

\begin{theorem} Let $X$ be a non-degenerate rank-1 system and $(v_n\,:\, n\in\N)$ be its canonical generating sequence. Then $(X,\sigma)$ is topologically isomorphic to its inverse $(X,\sigma^{-1})$ if and only if for all but finitely many $n\in\N$, $v_{n+1}$ is built symmetrically from $v_n$.
\end{theorem}

\subsection{Topological conjugacy of minimal systems}
The results of this paper, especially Theorem~\ref{rank1isom}, can also be viewed as a contribution to the study of the topological isomorphism problem for Bernoulli subshifts. In symbolic dynamics the problem is better known as the topological conjugacy problem. The objective is to understand the complexity of the topological isomorphism (or conjugacy) relation when it is restricted to important classes of Bernoulli subshifts. 

Clemens proved in \cite{Clemens1} that topological conjugacy for all Bernoulli subshifts is a universal countable Borel equivalence relation. Gao, Jackson and Seward (\cite{GaoJacksonSeward}), and independently Clemens (\cite{Clemens2}), also showed the same for free Bernoulli subshifts. However, the complexity of  topological conjugacy for minimal subshifts is an open question.  Gao, Jackson and Seward have shown in \cite{GaoJacksonSeward} that $E_0$ is a lower bound for this problem. Theorem~\ref{rank1isom} implies that topological conjugacy for minimal rank-1 systems has the same complexity as $E_0$. 

\subsection{Organization and terminology}
This paper is organized as follows. In Section 2 we define rank-1 words and rank-1 systems, and prove all the preliminary results about them. In particular we give a complete analysis of the collection of all finite words from which a rank-1 word is built and present several applications of this analysis. One of the applications is to define the canonical generating sequence for a rank-1 word.  Another is to define a complete metric on the space $\mathcal{R}$.  In Section 3 we prove the main results of this paper.

We denote the set of all natural numbers, including zero, by $\N$.  

Three types of words appear in this paper.  By {\em bi-infinite word} we mean an element of $\{0,1\}^\Z$.  We will frequently use lower case letters such as $x$, $y$, and $z$ to denote bi-infinite words.  By {\em infinite word} we mean an element of $\{0,1\}^\N$.  We will frequently use upper case letters such as $U$, $V$, and $W$ to denote infinite words.  By {\em finite word} we mean a finite sequence from the alphabet $\{0,1\}$.  We will frequently use lower case letters such as $u$, $v$, and $w$ to denote finite words.  

If $v$ and $w$ are finite words, $vw$ denotes the concatenation of $v$ and $w$.  Likewise, if $v$ is a finite word and $n \in \N$, then $v^n$ denotes the concatenation of $n$ copies of $v$.  We are principally interested in the finite words that begin and end with 0.  Throughout the rest of the paper $\mathcal{F}$ denotes the set of such finite words.

\section{The Polish space of non-degenerate rank-1 systems}

\subsection{Overview}
This section (Section 2) contains the foundational results about rank-1 words and systems that will allow us, in Section 3, to analyze isomorphisms between non-degenerate rank-1 systems and understand the complexity of the isomorphism relation on the space of such systems.  We now give an overview of what the rest of this section contains.

In Section \ref{rank1} we define a rank-1 word and show how it gives rise to a rank-1 system $(X, \sigma)$.  We show that there are three types of rank-1 systems: the system is degenerate if $X$ is finite; it is minimal if $X$ is a Cantor space on which the shift acts minimally; and it is non-minimal if $X$ is a Cantor space with a fixed point (i.e., a point that is fixed by the shift).

We define $\mathcal{R}$ to be the collection of rank-1 words that give rise to non-degenerate rank-1 systems and endow it with a natural topology.  As the map sending an element of $\mathcal{R}$ to the non-degenerate rank-1 system it gives rise to is one-to-one (this is Proposition \ref{proponetoone} and is proved at the very end of Section 2), we can view $\mathcal{R}$ as a space of codes for non-degenerate rank-1 systems.  We show that although there is another reasonable choice for a space of codes for non-degenerate rank-1 systems, these two spaces of codes are Borel isomorphic (with a natural isomorphism).

In Section \ref{secbuiltfrom} we analyze the collection $A_V$ of finite words from which a rank-1 word $V$ is built.  There is a natural partial order $\preceq$ on $A_V$.  We show that $(A_V, \preceq)$ is a lattice and that a certain natural subset of $A_V$ is linearly ordered by $\preceq$.  This certain subset of $A_V$ is infinite if and only if $V \in \mathcal{R}$ (i.e., if the system associated to $V$ is non-degenerate), and in the case that $V \in \mathcal{R}$ it gives rise to the canonical generating sequence of $V$.  (This sequence will be important in understanding the complexity of the isomorphism relation on $\mathcal{R}$ as a Borel equivalence relation.)  We then use our analysis of $A_V$ and this certain subset to define a complete metric on $\mathcal{R}$ that is compatible with its natural topology.  

In Section \ref{expected} we show that certain structural aspects of a non-degenerate rank-1 word $V$ imply similar structural aspects for each element of the non-degenerate rank-1 system $X$.  The key notion is that of an expected occurrence of a finite word $v$ in an infinite word $V$ (or a bi-infinite word $x$).  

If a $V \in \mathcal{R}$ is built from a finite word $v$, then there is a unique way to view $V$ as a collection of disjoint occurrences of $v$ interspersed with 1s; an occurrence of $v$ in $V$ is said to be {\em expected} if it is an element of that collection.  We show that this implies that each element $x$ of the rank-1 system $X$ associated to $V$ also has this property, i.e., it can be viewed in a unique way as a collection of disjoint occurrences of $v$ interspersed with 1s; an occurrence of $v$ in $x$ is {\em expected} if it is an element of that collection.  

The ability to analyze elements of a non-degenerate rank-1 system according to the location of their expected occurrences of $v$ will be an essential tool in the main technical proof of this paper (Theorem \ref{thmreplacement}).

\subsection{Rank-1 words and systems}
\label{rank1}

\subsubsection{Definition of a rank-1 word and its associated rank-1 system}
\begin{definition}  Let $V \in \{0,1\}^\N$.  
\begin{enumerate}
\item [ (a) ] $V$ is {\em built from} $v \in \mathcal{F}$ if there is a sequence $(a_i : i \geq 1)$ of natural numbers so that $V = v 1^{a_1} v 1^{a_2}v \ldots$.  
\item [ (b) ] Let $A_V = \{v \in \mathcal{F} : \textnormal{$V$ is built from $v$}\}$.  
\item [ (c) ] $V$ is {\em rank-1} if $A_V$ is infinite.
\end{enumerate}
\end{definition}  

Remarks:
\begin{enumerate}
\item  If $V$ is built from any $v \in \mathcal{F}$, then $V$ has infinitely many occurrences of 0.
\item  If $V$ is rank-1, then for each $n \in \N$, there is some $v \in A_V$ with $|v| \geq n$.
\item  If $V$ is rank-1, then for each $n \in \N$, there is some $v \in A_V$ with at least $n$ occurrences of 0.
\end{enumerate}

Let $V$ be a rank-1 word.  We construct a dynamical system $(X, \sigma)$ as follows.  Let $$X = \{x \in \{0,1\}^\Z : \textnormal{every finite subword of $x$ is a subword of $V$}\}$$  It is clear that $X \subseteq \{0,1\}^\Z$ is closed.  We define the shift $\sigma : X \rightarrow X$ by $\sigma (x) (i) = x(i+1)$.  It is clear that $\sigma$ is a homeomorphism of $X$.  We call $(X, \sigma)$ the {\em rank-1 system} associated to $V$.  

Though not necessary for the work done in this paper, we mention that there is a canonical measure associated to each rank-1 system $(X, \sigma)$; one defines $\mu (\{x  \in X : x (0) = 0\}) = 1$ and shows that this extends uniquely to a $\sigma$-finite, shift-invariant measure $\mu$ on $X$.  A few comments about this canonical measure will be made after the proof of Proposition \ref{prop1}.

\subsubsection{Three types of rank-1 words (and rank-1 systems)}

Some rank-1 words give rise to completely uninteresting rank-1 systems.  Proposition \ref{prop1} below shows that three distinct types of rank-1 words give rise to three distinct types of rank-1 systems.  First, however, we need a small proposition and a corollary.

\begin{proposition}
\label{propfirst}
Suppose $V$ is a rank-1 word and $(X, \sigma)$ the rank-1 system associated to $V$.  If $x \in X$ contains an occurrence of 0, then every finite subword of $V$ occurs in $x$.
\end{proposition}

\begin{proof}
Let $\alpha$ be a finite subword of $V$.  There is some $v$ from which $V$ is built that contains an occurrence of $\alpha$.  We know that every occurrence of 0 in $V$ is a part of an occurrence of $v$ in $V$.  It follows that every occurrence of 0 in any $x \in X$ is a part of an occurrence of $v$ in $x$.  Thus, if $x \in X$ contains an occurrence of 0, then it contains an occurrence of $v$ and, therefore, an occurrence of $\alpha$.
\end{proof}

\begin{corollary}
\label{cor}
Let $(X, \sigma)$ be a rank-1 system and suppose $x \in X$ contains an occurrence of 0.
\begin{enumerate}
\item [ (a) ] The orbit of $x$ is dense in $X$.
\item [ (b) ] There are infinitely many occurrences of 0 in $x$.
\end{enumerate}
\end{corollary}

\begin{proposition}
\label{prop1}
Let $V$ be a rank-1 word and $(X, \sigma)$ the rank-1 system associated to $V$.
\begin{enumerate}
\item  [ (a) ] If there is a word $v$ so that $V = vvv \ldots$, then $X$ is finite. 
\item  [ (b) ]  If there is no word $v$ so that $V = vvv \ldots$, but there is some $n\in \N$ so that $1^n$ is not a subword of $V$, then $(X, \sigma)$ is a Cantor minimal system.
\item  [ (c) ]  If for each $n \in \N$, the word $1^n$ is a subword of $V$, then $X$ is a Cantor space with exactly one periodic point, the bi-infinite word that is constantly 1.
\end{enumerate}
\end{proposition}

\begin{proof}
Proof of (a):  Suppose $V = vvv \ldots$  It is clear that if $x \in \{0,1\}^\Z$ is of the form $( \ldots v v v \ldots)$, then $x \in X$.  Since such an $x$ contains an occurrence of 0, we know by Corollary \ref{cor} that its orbit is dense.  But its orbit is finite, so $X$ must also be finite. 

Proof of (b):  Suppose that there is no word $v$ so that $V= vvv \ldots$, but there is an $n\in \N$ so that $1^n$ is not a subword of $V$.  

We claim that $X$ is non-empty.  Consider the bi-infinite word $y \in \{0,1\}^\Z$ defined by $y(i) = V(i)$ for $i \geq 0$ and $y(i) = 1$ for $i <0$.  Then consider the sequence of bi-infinite words $(\sigma^n (y) : n \in \N)$.  Since $\{0,1\}^\Z$ is compact, this sequence must have a limit point in $\{0,1\}^\Z$, and it is easy to check that this limit point must be in $X$.

We now know that $X$ is a non-empty, closed subset of $\{0,1\}^\Z$.  To show that $X$ is a Cantor space, it suffices to show that $X$ has no isolated points.  To show that the system $(X, \sigma)$ is minimal, it suffices to show that the orbit of each $x \in X$ is dense in $X$.  

Note that since each subword of $V$ of length $n$ contains an occurrence of 0, each subword of each $x \in X$ of length $n$ contains an occurrence of 0.  Thus, each $x \in X$ contains an occurrence of 0 and, by Corollary \ref{cor}, the orbit of each $x \in X$ is dense.

To show that $X$ has no isolated points, it suffices to show that if an element of $X$ has an occurrence of a finite word $\alpha$ beginning at $i$, then at least two distinct elements of $X$ have an occurrence of $\alpha$ beginning at $i$.  Suppose $x \in X$ has an occurrence of $\alpha$ beginning at $i$.  The finite word $\alpha$ must be a subword of $V$ and hence a subword of some $v$ from which $V$ is built.  Since $V$ is built from $v$, but $V$ is not periodic, we can find distinct $a, b \in \N$ so that $v1^a0$ and $v1^b0$ are subwords of $V$.  Since $x$ must contain an occurrence of 0, Proposition \ref{propfirst} implies that both $v1^a0$ and $v1^b0$ are subwords of $x$.  By applying appropriate powers of the shift to $x$, one gets two distinct elements of the orbit of $x$ that have an occurrence of $\alpha$ beginning at $i$.

Proof of (c):  Suppose that for each $n \in \N$, the word $1^n$ is a subword of $V$.  It is clear that $X$ contains the bi-infinite word that is constantly 1, and that this is a fixed point of $\sigma$.  We need to show that $X$ has no isolated points (and hence, is a Cantor space) and that every element of $X$ that has an occurrence of 0 is not periodic.  

Consider the bi-infinite word $z \in \{0,1\}^\Z$ defined by $z(i) = V(i)$ for $i \geq 0$ and $z(i) = 1$ for $i <0$. We claim that $z \in X$.  Every finite subword of $z$ that ends before position 0 is of the form $1^n$ and so is a subword of $V$.  Every finite subword of $z$ that begins after position -1 is clearly a subword of $V$.  It remains to show that each finite subword of $z$ that begins before 0 and ends after -1 is a subword of $V$.  It suffices to show that if $V$ is built from $v$, then for arbitrarily large $n \in \N$, the word $1^n v $ is a subword of $V$.  But if $V$ is built from $v$ and $n > |v|$, then $V$ must have an occurrence of $1^n$ and it must be followed by $v$ (this is because $V$ is built from $v$ and you cannot have the word $1^n$ be in the middle an occurrence of $v$, when $n > |v|$).  Thus $z \in X$.

We claim that $X$ has no isolated points.  Let $x \in X$.  To show that $x$ is not isolated, it suffices to show that if $x$ has an occurrence of a finite word $\alpha$ beginning at $i$, then there are two distinct elements of $X$ that have an occurrence of $\alpha$ beginning at $i$.  Suppose that $x$ has an occurrence of $\alpha$ beginning at $i$.  The finite word $\alpha$ must be a subword of $V$ and hence a subword of some $v$ from which $V$ is built.  Since $V$ is built from $v$ and each $1^n$ is a subword of $V$, we can find distinct $a, b \in \N$ so that $v1^a0$ and $v1^b0$ are subwords of $V$.  We know by Proposition \ref{propfirst} that both $v1^a0$ and $v1^b0$ are subwords of $z$.  There are thus two distinct elements of the orbit of $x$ that have an occurrence of $\alpha$ beginning at $i$.  Therefore, $x$ is not an isolated point.

We claim that every $x \in X$ that has an occurrence of 0 is not periodic.  Let $x \in X$ have an occurrence of 0.  By Corollary \ref{cor}, the orbit of $x$ is dense in $X$.  Since the bi-infinite word that is constantly 1 is in $X$ but not in the orbit of $x$, the orbit of $x$ must be infinite.  Therefore, $x$ is not periodic.
\end{proof}

Remarks:
\begin{enumerate}
\item  If a rank-1 word satisfies the hypothesis of (a), or a rank-1 system satisfies the conclusion of (a), then we say it is {\em degenerate}.  In this case, the canonical measure $\mu$ has atoms. 
\item  If a rank-1 word satisfies the hypothesis of (b), or a rank-1 system satisfies the conclusion of (b), then we say it is {\em minimal}.   In this case, the canonical measure $\mu$ is atomless and finite; normalizing, $(X, \mu, \sigma)$ is a measure-preserving transformation. 
\item  If a rank-1 word satisfies the hypothesis of (c), or a rank-1 system satisfies the conclusion of (c), then we say it is {\em non-minimal}. In this case, the canonical measure $\mu$ is atomless and could be finite or infinite.  
\end{enumerate}

\begin{corollary}
\label{cor1}
If $(X, \sigma)$ is a non-degenerate rank-1 system, then each $x \in X$ that contains an occurrence of 0 is not periodic.
\end{corollary}

\subsubsection{The space of non-degenerate rank-1 words}
\label{Setuplast}
Let $\mathcal{R}^\prime$ denote the set of rank-1 words and $\mathcal{R}$ denote the set of non-degenerate rank-1 words.  For $v \in \mathcal{F}$, let $$A^v = \{V \in \mathcal{R}^\prime : \textnormal{$V$ is built from $v$}\} = \{V \in \mathcal{R}^\prime : v \in A_V\}$$  Let $T^\prime$ be the topology on $\mathcal{R}^\prime$ generated by $\{A^v : v \in \mathcal{F}\}$.  Let $T$ be the topology that $\mathcal{R}$ inherits as a subset of $\mathcal{R}^\prime$ (with topology $T^\prime$).  Since $\mathcal{F}$ is countable, $(\mathcal{R}, T)$ is separable.  In Section \ref{secmetric}, we will describe a complete metric on $\mathcal{R}$ that is compatible with $T$; hence, $(\mathcal{R}, T)$ is a Polish space.  

There is, in fact, another way to see that $(\mathcal{R}, T)$ is Polish, which we now briefly describe.  One can check that $\mathcal{R}^\prime$ is an $F_{\sigma \delta}$ subset of $\{0,1\}^\N$.  There is a standard way to give a Polish topology to an $F_{\sigma \delta}$ subset of a Polish space; it involves refining the topology on the original Polish space in order to make the specified subset $G_\delta$ in the new topology, hence making it Polish (see, for example, \cite{Kechris} Lemmas 13.2 and 13.3).  One can check that in our case this procedure gives precisely the topology $T^\prime$ to $\mathcal{R}^\prime$.  Since $\mathcal{R}^\prime \setminus \mathcal{R}$ is countable, $\mathcal{R}$ is a $G_\delta$ subset of $\mathcal{R}^\prime$ and hence is Polish with the induced topology, which is $T$.

\subsubsection{The Borel isomorphism between two coding spaces}
Each non-degenerate rank-1 word is associated to a unique non-degenerate rank-1 system (and vice versa, by Proposition \ref{proponetoone}).  Thus we can view $\mathcal{R}$ as a (Polish) space of codes for all non-degenerate rank-1 systems.


There is, however, another standard coding for non-degenerate rank-1 systems as Bernoulli subshifts.  Consider
$$ K(\{0,1\}^\Z)=\{ X\subseteq \{0,1\}^\Z\,:\,\mbox{ $X$ is a closed (compact) subset of $\{0,1\}^\Z$}\}. $$
With the Hausdorff metric or the Vietoris topology, $K(\{0,1\}^\Z)$ becomes a Polish space (see, for example, \cite{Kechris} Theorem 4.25). Consider further
$$ \mathcal{S}=\{X\in K(\{0,1\}^\Z)\,:\, \mbox{$X$ is invariant under the shift}\}. $$
Then $\mathcal{S}$ is a closed subspace of $K(\{0,1\}^\Z)$, and therefore also a Polish space. Intuitively, the space $\mathcal{S}$ is the space of all Bernoulli subshifts. Let $\mathcal{R}^*$ be the subspace of $\mathcal{S}$ consisting of all non-degenerate rank-1 systems. Then we will verify below that $\mathcal{R}^*$ is a Borel subset of $\mathcal{S}$, and therefore a standard Borel space.

To see that $\mathcal{R}^*$ is Borel, consider the map $\Phi: \mathcal{R}\to\mathcal{S}$ defined as
$$ \Phi(V)=\{x\in\{0,1\}^\Z\,:\, \mbox{every subword of $x$ is a subword of $V$}\}, $$
and note that $\mathcal{R}^*$ is exactly the range of $\Phi$.
It is easy to verify that $\Phi$ is a Borel map. By Proposition~\ref{proponetoone} it is also one-to-one. By a theorem of Luzin--Suslin (see, for example, \cite{Kechris} Theorem 15.1 and Corollary 15.2), we conclude that $\mathcal{R}^*$ is Borel and that $\Phi$ is in fact a Borel isomorphism between $\mathcal{R}$ and $\mathcal{R}^*$.

\subsection{Finite words from which a specified rank-1 word is built}
\label{secbuiltfrom}
We will be interested in the structure of the set of finite words from which a specified rank-1 word is built.    

\begin{definition}
Let $v, w \in \mathcal{F}$. 
\begin{enumerate}
\item [ (a) ]  We say $w$ is {\em built from} $v$, and write $v \preceq w$, if there exists $r \geq 1$ and $a_1, a_2, \ldots a_{r-1} \in \N$ so that $w = v 1^{a_1} v 1^{a_2} v \ldots v 1^{a_{r-1}} v$.  
\item [ (b) ] If, moreover, $a_1 = a_2 = \ldots = a_{r-1}$, then we say $w$ is {\em built simply} from $v$ and write $v \preceq_s w$.
\end{enumerate}
\end{definition}

When $v \preceq w$, but $v \neq w$, we write $v \prec w$.  It is easy to check that that the relation $\preceq$ is a partial order on $\mathcal{F}$, i.e., it is reflexive, anti-symmetric, and transitive.  While the relation $\preceq_s$ is not transitive, we do have the following simple but useful lemma, which we state without proof.  Figure 1 illustrates the lemma.

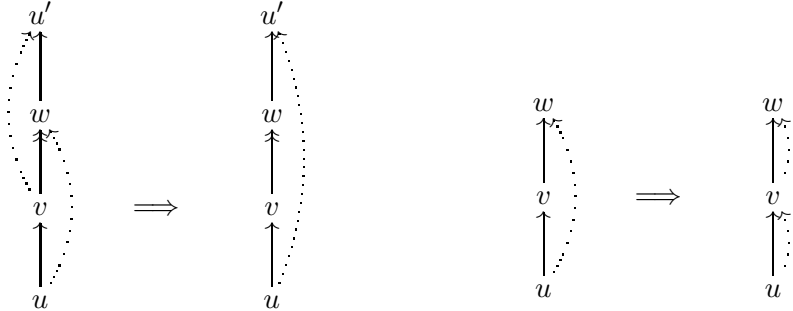
\begin{figure}[b]
$\xymatrix{
u' 		&  & u'\\ 
w \ar[u] &  & w\ar[u]\\ 
v \ar@{->>}[u] \ar@{.>}@/^1pc/[uu] &  \Longrightarrow  & v\ar@{->>}[u]\\ 
u \ar[u] \ar@{.>}@/_1pc/[uu] &  &u\ar[u] \ar@{.>}@/_1pc/[uuu]}$ 
\ \ \ \ \ \ \ \ \ \ \ \ \ \ \ \ \ \ \ \ \ \ \ 
$\xymatrix{& & \\ 
w & & w \\ 
v \ar[u] & \Longrightarrow & v \ar[u] \ar@{.>}@/_/[u] \\ 
u \ar[u] \ar@{.>}@/_1pc/[uu] & & u\ar[u] \ar@{.>}@/_/[u]}$

\caption{Lemma~\ref{lemmasimply} (a) and (b): solid arrows stand for $\preceq$, 
two-head arrows for $\prec$, and 
dotted arrows for $\preceq_s$.
}
\end{figure}

\begin{lemma}
\label{lemmasimply}
Let $u, v,w, u^\prime \in U$.   
\begin{enumerate}
\item [ (a) ]  If $ u \preceq v \prec w \preceq u^\prime$ and $u \preceq_s w$ and $v \preceq _s u^\prime$, then $u \preceq_s u^\prime$.
\item [ (b) ]  If $u \preceq v \preceq w$ and $u \preceq_s w$, then $u \preceq_s v$ and $v \preceq_s w$.
\end{enumerate}
\end{lemma}

\subsubsection{The lattice $(A_V, \preceq)$}  
\label{seclattice}
For the rest of Section \ref{seclattice}, fix a rank-1 word $V$.  Recall that $A_V = \{v \in \mathcal{F}: \textnormal{$V$ is built from $v$}\}$.  Since $V$ is rank-1, $A_V$ is infinite.  Since $(\mathcal{F}, \preceq)$ is a partial order, so is $(A_V, \preceq)$.  

Our next objective is to show that $(A_V, \preceq)$ is a lattice.  Recall that a partial order is a lattice if every two elements have a least upper bound and a greatest lower bound.  The greatest lower bound of of two elements $v$ and $w$ is called the {\em meet} of $v$ and $w$ and is denoted by $v \wedge w$.  The least upper bound of $v$ and $w$ is called the {\em join} and is denoted by $v \vee w$.

Let $(i_n : n \in \N$) enumerate, in an order-preserving way, the positions of the occurrences of 0 in $V$.  Since $V$ is rank-1, $V$ begins with an occurrence of 0.  Therefore, $i_0 = 0$.  Define $L_V: \N_{>0} \rightarrow \N$ by $L_V(n) = i_n - i_{n-1} -1$.  Note that $L_V(n)$ is the number of 1s in $V$ between the occurrence of 0 at position $i_{n-1}$ and the occurrence of 0 at position $i_n$.  When $V$ is clear from the context, we write $L$ for $L_V$.

We state the following Lemma without proof.

\begin{lemma}
\label{lemmacharbuilt}
Let  $v = 0 1^{a_1} 0 1^{a_2} \ldots 1^{a_{r-1} }0 $ and $w =  0 1^{b_1} 0 1^{b_2} \ldots 1^{b_{s-1} }0 $.  
\begin{enumerate}
\item [ (a) ]  $v \in A_V$ iff for $0 < k < r$ and $i \equiv k \mod r$, $L_V(i) = a_k$.
\item [ (b) ]  If $v,w \in A_V$, then $v \preceq w$ iff $s$ is a multiple of $r$.
\end{enumerate}
\end{lemma}

We now show that $(A_V, \preceq)$ is a lattice.

\begin{proposition}
\label{proplattice}
If $v, w \in A_V$, then there exist $u, u^\prime \in A_V$ such that
\begin{enumerate}
\item [ (a) ] $u \preceq v \preceq u^\prime$ and $u \preceq w \preceq u^\prime$;
\item [ (b) ] if $\alpha \in A_V$ with $v \preceq \alpha $ and $w \preceq \alpha $, then $u^\prime \preceq \alpha$; and
\item [ (c) ] if $\alpha \in A_V$ with $\alpha \preceq v$ and $\alpha \preceq w$, then $\alpha \preceq u$.
\end{enumerate}
\end{proposition}

\begin{proof}
Let $v = 0 1^{a_1} 0 1^{a_2} \ldots 01^{a_{r-1}}0$ and $w =  0 1^{b_1} 0 1^{b_2} \ldots 1^{b_{s-1} }0$ and suppose that $v, w \in A_V$.

We will define $u$ and $u^\prime$ and show that $u,u^\prime \in A_V$.  Let $t = \gcd(r,s)$.  For $0 < k < t$, let $c_k = L(k)$.  Now let $u = 0 1^{c_1} 0 1^{c_2}0 \ldots 01^{c_{t-1} }0$.  Let $t^\prime = \operatorname{lcm}(r,s)$.  For $0 < k < t^\prime$, let $d_k = L(k)$.  Now let $u^\prime = 0 1^{d_1} 0 1^{d_2}0 \ldots 01^{d_{t^\prime-1} }0$.

We claim that $u \in A_V$.   By part (a) of Lemma \ref{lemmacharbuilt}, it suffices to show that for $0 < k < t$ and $i \equiv k \mod t$, $L(i) = L(k)$.  Suppose $i \equiv k \mod t$ and $0 < k < t$.  Choose $l \in \N$ so that $i = k + lt$.  Since $t = \gcd (r,s)$, there are $m,n \in \N$ such that $t = mr - ns$.  This implies that $i + lns = k + lmr$.  Since $i \not\equiv 0 \mod s$, we know that $L$ is constant on the congruence class of $i$ (mod $s$).  Thus, $L(i) = L(i + lns)$.  Since $k \not\equiv 0 \mod r$, we know that $L$ is constant on the congruence class of $k$ (mod $r$).  Thus, $L(k) = L(k + lmr)$.  Since $i + lns = k + lmr$, we have that $L(i) = L(k)$.

 We claim that $u^\prime \in A_V$.  By part (a) of Lemma \ref{lemmacharbuilt}, it suffices to show that for $0 < k < t^\prime$ and $i \equiv k \mod t^\prime$, $L(i) = L(k)$.  Suppose $i \equiv k \mod t^\prime$ and $0 < k < t^\prime$.  Since $k \not\equiv 0 \mod t^\prime$ and $t^\prime = \operatorname{lcm} (r,s)$, either $k \not\equiv 0 \mod r$ or $k \not\equiv 0 \mod s$.  If $k \not\equiv 0 \mod r$, then $L$ is constant on the congruence class of $i$ (mod $r$).  Thus, $L(i) = L(k)$.  If $k \not\equiv 0 \mod s$, then $L$ is constant on the congruence class of $i$ (mod $s$).  Thus, $L(i) = L(k)$.  

The proofs of (a), (b), and (c) follow immediately from the definition of $u$ and $u^\prime$ and part (b) of Lemma \ref{lemmacharbuilt}.
\end{proof}

We have one more proposition whose proof is closely connected with the proof of Proposition \ref{proplattice}. 

\begin{proposition} 
\label{propincomparable}
 If $v, w \in A_V$ are incomparable, then:
\begin{enumerate}
\item [ (a) ] $(v \wedge w) \prec v \prec (v \vee w)$;
\item [ (b) ] $(v \wedge w) \prec w \prec (v \vee w)$; and
\item [ (c) ] $(v \wedge w) \preceq_s (v \vee w)$.
\end{enumerate}
If, moreover, $V = v 1^a v 1^a v \ldots$, then $V = (v \wedge w) 1^a (v \wedge w) 1^a (v \wedge w) \ldots$.
\end{proposition}

\begin{proof}  (a) and (b) are obvious.  We now prove (c).

Let $v, w \in A_V$ be incomparable.  We adopt the notation given in the proof of Proposition \ref{proplattice}.  To show that $u^\prime = v \vee w$ is built simply from $u = v \wedge w$, it suffices to show that if $0 < i < t^\prime$ and $i \equiv 0 \mod t$, then $L(i) = L(t)$.  

Since $v$ and $w$ are incomparable, $t \neq r$ and $t \neq s$.  Suppose that  $0 < i < t^\prime$ and $i \equiv 0 \mod t$.  We want to show that $L(i) = L(t)$.  We know that $i \not\equiv 0 \mod t^\prime$.  Therefore, since $t^\prime = \operatorname{lcm} (r,s)$, we have that either $i \not\equiv 0 \mod r$, or $i \not\equiv 0 \mod s$.  Without loss of generality, assume that $i \not\equiv 0 \mod s$. 

Suppose $i \not\equiv 0 \mod s$.  Let $l \in \N$ be such that $i = t + lt$.  Let $m,n \in \N$ be such that $t = mr -ns$.  This implies that $i + lns = t + lmr$.  Since $i \not\equiv 0 \mod s$, we know that $L$ is constant on the congruence class of $i \mod s$.  Thus, $L(i) = L(i + lns)$.  Since $t = \gcd (r,s)$, but $t \neq r$, we know that $t \not\equiv 0 \mod r $, which implies that $L$ is constant on the congruence class of $t \mod r$.  Thus, $L(t) = L(t + lmr)$.  Since $i + lns = t + lmr$, we have that $L(i) = L(t)$.

We have now proved (c), i.e., that $(v \wedge w) \preceq_s (v \vee w)$.    

Now suppose that $V = v 1^a v 1^a \ldots$.  To show that $V =(v \wedge w) 1^a (v \wedge w)1^a \ldots$, it suffices to show that if $i \equiv 0 \mod t$, then $L(i) = a$.

We know that if $i \equiv 0 \mod r$, then $L(i) = a$.  Thus, if $i \equiv 0 \mod t^\prime$, then $i \equiv 0 \mod r$ (since $t^\prime$ is a multiple of $r$) and hence, $L(i) =a$.

But we also know that $u \preceq_s u^\prime$, and this implies that $L(i)$ is the same for all $i$ satisfying $i \equiv 0 \mod t$ and $0 < i < t^\prime$.  Since $V$ is built from $u^\prime$, $L(i)$ is the same for all $i$ satisfying $i \equiv 0 \mod t$ and $i \not\equiv0 \mod t^\prime$.  We know that $r \equiv 0 \mod t$ and $r \not\equiv 0 \mod t^\prime$ and $L(r) = a$.  Thus, if $i \equiv 0 \mod t$ and $i \not\equiv 0 \mod t^\prime$, then $L(i) = a$.

Together, the last two paragraphs show that if $i \equiv 0 \mod t$, then $L(i) = a$.
\end{proof}

\begin{corollary}
\label{cordegeneratesimple}
Suppose $V$ is degenerate and that $V= v1^a v1^a v \ldots$, with $|v|$ as small as possible.  If $w \in A_V$ and $|v| \leq |w|$, then $w = v1^a v1^a v \ldots v 1^a v$ and, therefore, $v \preceq_s w$. 
\end{corollary}

\subsubsection{The total order $(B_V, \preceq)$}
\label{sectotalorder}

\begin{definition}  Let $V \in \{0,1\}^\N$.  
\begin{enumerate}
\item [ (a) ] $V$ is {\em fundamentally} built from $v$ if $v \in A_V$ and for all $u,u^\prime \in A_V$ with $u \prec v \prec u^\prime$, $u \not\preceq_s u^\prime$.
\item [ (b) ] Let $B_V = \{v \in U: \textnormal{$V$ is fundamentally built from $v$}\}$.  
\end{enumerate}
\end{definition}  

Note that if $V\in \mathcal{R}^\prime$, then $0\in B_V$.

For the rest of Section \ref{sectotalorder}, fix a rank-1 word $V$.  We have two important propositions about $B_V$.

\begin{proposition}
\label{propsubset}
Every element of $B_V$ is comparable to every element of $A_V$.
\end{proposition}

\begin{proof}
Suppose, towards a contradiction that $v \in B_V$ and $w \in A_V$ are not comparable.  Then, by Proposition \ref{propincomparable}, $(v \wedge w) \prec v \prec (v \vee w)$ with $(v \wedge w) \preceq_s (v \vee w)$.  This is a contradiction with $v \in B_V$.  
\end{proof}

\begin{corollary}
\label{corBtotalorder}
$(B_V, \preceq)$ is a total order.
\end{corollary}

\begin{proposition}
\label{propBinfinite}
The set $B_V$ is infinite iff $V \in \mathcal{R}$ (i.e., iff $V$ is non-degenerate).
\end{proposition}

\begin{proof}
Suppose $V$ is degenerate.  Let $v$ be of minimal length such that for some $a \in \N$, $V = v 1^a v 1^a v \ldots$.  To show $B_V$ is finite, it suffices to show that if $w \in A_V$ with $|v| < |w|$, then $w \notin B_V$.  Suppose $w \in A_V$ with $|v| < |w|$.  By Corollary \ref{cordegeneratesimple}, we know that $w = v1^a v1^a v \ldots v 1^a v$.  It is clear that $v \prec w \prec w1^a w$ and $v \preceq_s w1^aw$ and $w1^aw \in A_V$.  Therefore, $w \notin B_V$. 

Now suppose that $V$ is non-degenerate.  To show $B_V$ is infinite it suffices to show that for any $v \in A_V$, there is some $w \in B_V$ with $v \preceq w$.  Let $v \in A_V$ and choose $w \in A_V$ to be maximal (with respect to $\preceq$) so that $v \preceq_s w$.  The existence of such a $w$ exists follows from the fact that $V$ is not periodic.    

We claim that $w \in B_V$.  Suppose, towards a contradiction, that there exist $u, u^\prime \in A_V$ such that $u \prec w \prec u^\prime$ and so that $u \preceq_s u^\prime$.  If $v = w$, then we $u \prec v \prec u^\prime$ and $u \preceq_s u^\prime$.  By the second part of Lemma \ref{lemmasimply}, this implies $v \preceq_s u^\prime$, which contradicts the maximality of $w$.  

Suppose, then, that $v \prec w$.  If $v$ is comparable to $u$, then either $ u \preceq v \prec w \prec u^\prime $ or $v \preceq u \prec w \prec u^\prime$.  Either way, Lemma \ref{lemmasimply} implies that $ v \preceq_s u^\prime$, which contradicts the maximality of $w$.

Suppose, then, that $u$ and $v$ are incomparable.  The situation is illustrated below in Figure~\ref{propBinfinitefig}. By Proposition \ref{propincomparable}, we know that $(v \wedge u) \prec u \prec (v \vee u)$ and $(v \wedge u) \preceq_s (v \vee u)$.  Thus we have $(v \wedge u) \prec u \prec (v \vee u) \preceq u^\prime$, with $(v \wedge u) \preceq_s (v \vee u)$ and $u \preceq_s u^\prime$.  By the first part of Lemma \ref{lemmasimply}, $(v \wedge u) \preceq_s u^\prime$.  We now have $(v \wedge u) \prec v \prec w \prec u^\prime$, with $(v \wedge u) \preceq_s u^\prime$.  By the second part of Lemma \ref{lemmasimply},  $v \preceq_s u^\prime$, which contradicts the maximality of $w$. 

\begin{figure}
$\xymatrix{
&u'&  &  & &u'&  \\
&w\ar@{->>}[u]& &  & &w\ar@{->>}[u]&  \\
&v\vee u\ar[u]& & \Longrightarrow & &v\vee u\ar[u]&  \\
u\ar@{.>}@/^1pc/[uuur]\ar@{->>}[ur]& &v \ar@{->>}[ul]& & u\ar@{->>}[ur]& & v \ar@{->>}[ul]\ar@{.>}@/_/[uuul] \\
&v\wedge u \ar@{.>}@/_/[uu]\ar@{->>}[ul] \ar@{->>}[ur]& & & & v\wedge u \ar@{->>}[ul]\ar@{->>}[ur]& }$
\caption{\label{propBinfinitefig}The proof of Proposition~\ref{propBinfinite}.}
\end{figure}
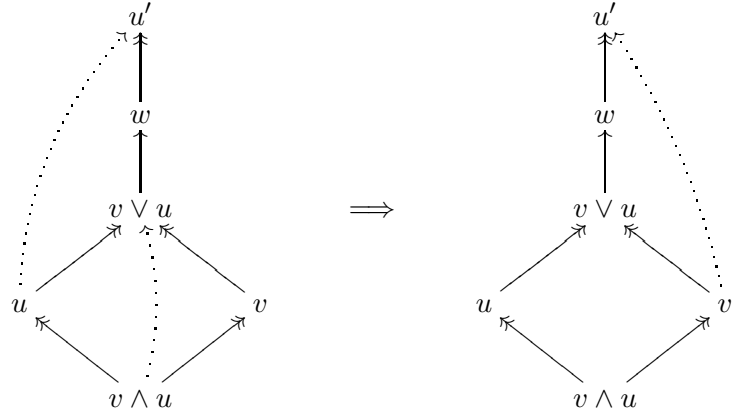

\end{proof}

In fact, the proof above showed the following.

\begin{corollary}
\label{corC}
If $V$ is non-degnerate and $v \in A_V$, then there is some $w \in B_V$ with $v \preceq w$.
\end{corollary}

\subsubsection{The canonical generating sequence of a non-degenerate rank-1 word}
\label{seccgs}

We now know that for $V \in  \mathcal{R}$, $B_V$ is an infinite set that is totally ordered by $\preceq$.  Let $(v_n : n \in \N)$ enumerate $B_V$ so that for $n, m \in \N$, $n \leq m$ iff $v_n \preceq v_m$.  We call the sequence $(v_n : n \in \N)$ the {\it canonical generating sequence} for $V$.

\subsubsection{Sets of the form $A^v$ and sets of the form $B^v$}

Recall that for $v \in \mathcal{F}$, $$A^v = \{V \in  \mathcal{R}^\prime : \textnormal{$V$ is built from $v$}\} = \{V \in  \mathcal{R}^\prime : v \in A_V\}.$$  
Let $$B^v = \{V \in  \mathcal{R}^\prime : \textnormal{$V$ is fundamentally built from $v$}\} = \{V \in  \mathcal{R}^\prime : v \in B_V\}.$$

Remarks:
\begin{enumerate}
\item  $ \mathcal{R}^\prime = A^0 = B^0$.
\item  For any $v$, $B^v \subseteq A^v$.
\item  If $v \preceq w$, then $A^w \subseteq A^v$.
\item  If $B^v \cap B^w \neq \emptyset$, then $v$ is comparable to $w$.
\end{enumerate}

\begin{proposition}
\label{propultrametric}  
If $v \prec w$ and $B^v \cap A^w \neq \emptyset$, then $A^w \subseteq B^v$.
\end{proposition}

\begin{proof}
Suppose $v \prec w$ and $V \in B^v \cap A^w$.

We claim that $A^w \subseteq B^v$.  Suppose, towards a contradiction, that $W\in A^w$, but $W \notin B^v$.  Then $w \in A_W$, but $v \notin B_W$.  Note that $v \in A_W$, since $W \in A^w \subseteq A^v$.  Since $v \notin B_W$, there are $u, u^\prime \in A_W$ so that $u \prec v \prec u^\prime$ and $u \preceq_s u^\prime$.  Note that $u \in A_V$ (since $u \prec v$ and $v \in A_V$), but that $u^\prime$ need not be an element of $A_V$.  We have three possibilities; either $u^\prime \preceq w$, $w \preceq u^\prime$, or $u^\prime$ and $w$ are incomparable.

If $u^\prime \preceq w$, then $u^\prime \in A_V$.  Thus, $u,u^\prime \in A_V$, with $u \prec v \prec u^\prime$ and $u \preceq_s u^\prime$.  This is a contradiction with $v \in B_V$.

If $w \preceq u^\prime$, then $u \prec v \prec w \preceq u^\prime$, with $u \preceq _s u^\prime$.  By Lemma \ref{lemmasimply}, this implies that $u \preceq_s w$.  Since $u, w \in A_V$, this is a contradiction with $v \in B_V$.

Suppose that $u^\prime$ and $w$ are incomparable. The situation is illustrated below in Figure~\ref{propultrametricfig}. Since $u^\prime$ and $w$ are both in $A_W$, we know $(u^\prime \vee w), (u^\prime \wedge w) \in A_W$.  We now have $u \prec v \preceq (u^\prime \wedge w) \prec u^\prime \prec (u^\prime \vee w)$, with $u \preceq_s u^\prime$ and $(u^\prime \wedge w) \preceq_s  (u^\prime \vee w)$.  By the first part of Lemma \ref{lemmasimply}, we know that $u \preceq_s (u^\prime \vee w)$.  Now we have $u \prec v \prec w \prec (u^\prime \vee w)$, with $u \preceq _s (u^\prime \vee w)$.  By the second part of Lemma \ref{lemmasimply}, we know that $u \preceq_s w$. Since $u, w \in A_V$, this is a contradiction with $v \in B_V$.
\end{proof}

\begin{figure}
$\xymatrix{
&u'\vee w& &  & &u'\vee w&  \\
u'\ar@{->>}[ur]& &w \ar@{->>}[ul]& & u'\ar@{->>}[ur]& & w \ar@{->>}[ul]\\
&u'\wedge w \ar@{.>}@/_/[uu]\ar@{->>}[ul] \ar@{->>}[ur]& & \Longrightarrow& & u'\wedge w \ar@{->>}[ul]\ar@{->>}[ur]& \\
&v\ar[u]&  &  & &v\ar[u]&  \\
&u\ar@{.>}@/^1pc/[uuul]\ar@{->>}[u]& &  & &u\ar@{->>}[u]\ar@{.>}@/_/[uuur] &  \\
 }$
\caption{\label{propultrametricfig}The proof of Proposition~\ref{propultrametric}. }
\end{figure}
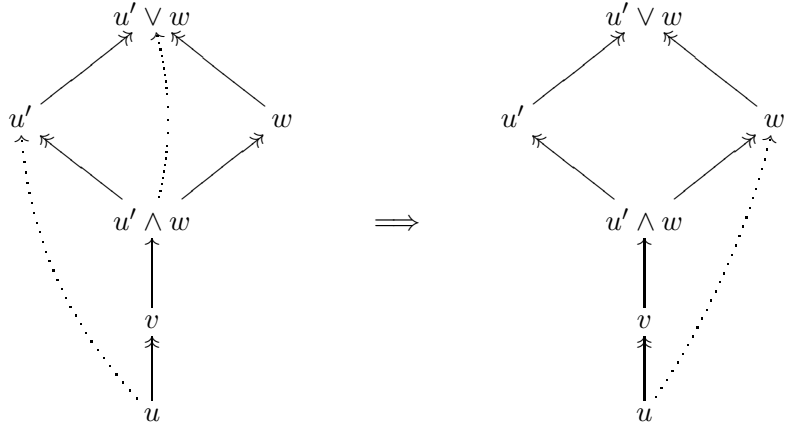

\begin{corollary}
\label{corultrametric}  
If $v \preceq w$ and $B^v \cap B^w \neq \emptyset$, then $B^w \subseteq B^v$.
\end{corollary}

\begin{proposition}
\label{propcompatible}
The topology generated by $\{A^v : v \in \mathcal{F}\}$ is the same as the topology generated by $\{B^v : v \in \mathcal{F}\}$.
\end{proposition}

\begin{proof}
It suffices to show the following:
\begin{enumerate}
\item [ (a) ] If $v \in \mathcal{F}$ and $V \in A^v$, then there exists $w \in \mathcal{F}$ so that $V \in B^w \subseteq A^v$.
\item [ (b) ] If $v \in \mathcal{F}$ and $V \in B^v$, then there exists $w \in \mathcal{F}$ so that $V \in A^w \subseteq B^v$.
\end{enumerate}

Suppose $v \in \mathcal{F}$ and $V \in A^v$. Choose $w \in B_V$ so that $v \preceq w$.  Clearly, $B^w \subseteq A^w \subseteq A^v$.  Since $w \in B_V$, $V \in B^w$.  This proves (a).

Suppose $v \in \mathcal{F}$ and $V \in B^v$.  Choose $w \in B_V$ so that $v \preceq w$.  Clearly, $B^v \cap A^w \neq \emptyset$.  By Proposition \ref{propultrametric}, $A^w \subseteq B^v$.  Since $w \in B_V$, $V \in B^w \subseteq A^w$.  This proves (b).
\end{proof}

\subsubsection{A complete metric on $ \mathcal{R}$}
\label{secmetric}

Define $d :  \mathcal{R} \times  \mathcal{R} \rightarrow \R$ by $$d(V,W) = 2^{-\sup \{|v| : V, W \in B^v\}}.$$

Remarks:
\begin{enumerate}
\item  If $V, W \in  \mathcal{R}$, then $V,W \in B^0$ and, therefore, $0 \leq d(V, W) \leq \frac{1}{2}$.
\item  By Proposition \ref{propcompatible}, $d$ generates the topology $T$ on $ \mathcal{R}$.
\end{enumerate}

\begin{proposition}
\label{propcomplete}
$( \mathcal{R}, d)$ is a complete ultrametric space.
\end{proposition}

\begin{proof}
We first show that $d$ is an ultrametric. For this we need to verify the following conditions for all $U,V,W \in  \mathcal{R}$.

(1) $d(U, V) \geq 0$.  This is immediate from the definition of $d$.

(2) $d(U, V) = d(V, U)$.  This is immediate from the definition of $d$.

(3) $d(U, V) = 0$ iff $U=V$.  If $V \in  \mathcal{R}$, then $B_V$ is infinite; thus, $d(V,V) = 0$.  If $d(U, V) = 0$, then there exist arbitrarily long words that are common initial segments of $U$ and $V$; thus, $U=V$.  

(4) For all $U,V,W \in  \mathcal{R}$, $d(U, W) \leq \max \{d(U,V), d(V,W)\}$.  This is clear if $U=V$ or $V = W$.  Otherwise, let $v$ be as long as possible satisfying $U, V \in B^v$ and let $w$ be as long as possible satisfying $V, W \in B^w$.  Without loss of generality, assume $|v| \leq |w|$.  Since $ B^v \cap B^w \neq \emptyset$, we know $v \preceq w$.  By Corollary \ref{corultrametric}, we know that $B^w \subseteq B^v$ and hence, $W \in B^v$.  Now $U, W \in B^v$ and thus, $d(U,W) \leq d(U, V)$.

Next we prove that $d$ is complete. Suppose $\{V_n\}$ is a $d$-Cauchy sequence in $ \mathcal{R}$.  Let $$B = \{v : \textnormal{for sufficiently large $n$, $V_n \in B^v$}\}.$$  

We claim that $(B, \preceq)$ is a total order.  If $v, w \in B$, then for sufficiently large $n$, $V_n \in B^v \cap B^w$.  Thus, $B^v \cap B^w \neq \emptyset$ and therefore, $v$ and $w$ are comparable. 

We claim that $B$ contains arbitrarily long $v$.  Let $M \in \N$.  Find $N$ so that for $n,m \geq N$, the distance between $V_n$ and $V_m$ is at most $2^{-M}$.  So for $n,m\geq N$, there is some $v$ of length at least $M$ so that $V_n, V_m \in B^v$.  Choose $v$ of minimal length so that $|v|\geq M$ and so that for some $n \geq N$, $V_n \in B^v$.  It follows from Corollary \ref{corultrametric} and the minimality of $|v|$ that if $m \geq N$, then $V_m \in B^v$.  Thus, $v$ is an element of $B$.

We now have $B$, an infinite set on which $\preceq$ is a total order.  Let $V$ be the unique infinite word such that each element of $B$ is an initial segment of $V$.  We claim that $V$ is rank-1.  It suffices to show that $B \subseteq A_V$.  If $v \in B$, then there are arbitrarily long initial segments of $V$ that are built from $v$; namely, the elements of $B$ that are longer than $v$.  Thus, $v \in A_V$.

We claim that $V$ is not degenerate.  By Proposition \ref{propBinfinite}, It suffices to show that $B \subseteq B_V$.  Let $v \in B$ and choose $w \in B$ with $v \prec w$.  Choose $n \in \N$ large enough that $V_n \in B^v \cap B^w$.  Clearly, $B^v \cap A^w \neq \emptyset$.  By Proposition \ref{propultrametric}, $A^w \subseteq B^v$.  Since $V \in A^w$, $V \in B^v$.  Therefore, $v \in B_V$.      

We claim that $V$ is the $d$-limit of the sequence $\{V_n\}$.  It suffices to show that $B_V \subseteq B$.  Let $v \in B_V$.  Choose $w \in B$ so that $v \prec w$.  Since $B \subseteq B_V$, $w \in B_V$, hence $B^v \cap B^w \neq \emptyset$.  By Corollary \ref{corultrametric}, $B^w \subseteq B^v$.  For sufficiently large $n \in \N$, we know $V_n \in B^w \subseteq B^v$.  Therefore, $v \in B$.
\end{proof}

\subsection{Expectedness in rank-1 words and elements of rank-1 systems}
\label{expected}

\subsubsection{Expectedness in a rank-1 word}
\label{secexpword}

For the rest of Section \ref{secexpword}, fix $V \in  \mathcal{R}$ and fix $v \in A_V$.  Let $(X, \sigma)$ be the rank-1 system associated to $V$ and let $v = 0 1^{a_1} 0 1^{a_2} \ldots 1^{a_{r-1}} 0$.

If $V(i) = 0$ and $|\{j < i : V(j) =0\}|$ is a multiple of $r$, then $V$ has an occurrence of $v$ beginning at $i$.  Such an occurrence of $v$ in $V$ is called {\em expected}.  Since $V$ is built from $v$, each occurrence of 0 in $V$ is part of exactly one expected occurrence of $v$.  It is possible for an occurrence of $v$ in $V$ to be {\em unexpected}; it is easy to see that such an occurrence of $v$ must overlap exactly two expected occurrence of $v$ in $V$. 

In Section \ref{secexpconnect} below we will show that if $x \in X$, then there is a unique collection of occurrences of $v$ in $x$ (these will be called the expected occurrences of $v$ in $x$) so that each occurrence of 0 is part of exactly one element of that collection.  We will also show that for any $i \in \Z$, the set of all $x \in X$ that have an expected occurrence of $v$ beginning at $i$ is a clopen subset of $X$.  The key to these results is noticing that one can determine whether an occurrence of $v$ in $V$ (say it begins at $i$) is expected by knowing what $V$ looks like ``close" to position $i$; in particular, we do not need to count the number of 0s in $V$ before position $i$.  This is formalized in Corollary \ref{corexp} below.

Recall that $(i_n : n \in \N$) enumerates, in an order-preserving way, the positions of the occurrences of 0 in $V$.  Recall also that $L: \N_{>0} \rightarrow \N$ is defined by $L(n) = i_n - i_{n-1} -1$.  If $V$ is built from $v$, then we know from part (a) of Lemma \ref{lemmacharbuilt} that for $0<k<r$ and $n \equiv k \mod r$, then $L(n) = a_k$.  Thus, for each $0 < k < r$, $L$ is constant on each congruence class of $k$ (mod $r$).  Since $V \in R$ is not periodic, neither is $L$.  Thus $L$ is not constant on the congruence class of 0 (mod $r$).  In fact, we can show that if $V$ is built from $v$, the function $L$ fails to be constant on the congruence class of $0$ (mod $r$) somewhat regularly.

\begin{proposition}
\label{propconsec}
There is some $t \in \N$ so that no $t$ consecutive elements of the sequence $(L(r), L(2r), L(3r), \ldots )$ are equal.
\end{proposition}

\begin{proof}
First we show that for sufficiently large $d$, there cannot even be two consecutive elements of $(L(r), L(2r), L(3r), \ldots )$ that each equal $d$.  Simply choose $w$ so that $V$ is built from $w$ and so that $w$ has more than $r$ occurrences of 0.  Let $D$ be the maximal number of consecutive 1s in $w$.  Since $w$ has more occurrences of 0 than $v$ does, every occurrence of $w$ in $V$ intersects at least two expected occurrences of $v$.  This implies that for every two consecutive elements of $(L(r), L(2r), L(3r), \ldots )$, at least one must be less than or equal to $D$.  Therefore, if $d >D$, there cannot be two consecutive elements of $(L(r), L(2r), L(3r), \ldots )$ that each equal $d$.  

Suppose, towards a contradiction, that there is no bound on the number of consecutive elements of $(L(r), L(2r), L(3r), \ldots )$ that are equal.  By the preceding paragraph, there must be a single value $d$ so that there is no bound on the number of consecutive elements of $(L(r), L(2r), L(3r), \ldots )$ that equal $d$.  Now let $\alpha = v1^d$ and notice that for each $n$, $\alpha^n$ is a subword of $V$.  This implies the existence of an $x \in X$ that contains a 0 and that is periodic, namely, $(\ldots \alpha \alpha \alpha \ldots )$.  This contradicts Corollary \ref{cor1}.
\end{proof}

For the rest of Section \ref{secexpword}, fix $t \in \N$ so that no $t$ consecutive elements of the sequence $(L(r), L(2r), L(3r), \ldots )$ are equal.  

\begin{corollary}
\label{corexpV}
The following statements (about $n \in \N$) are equivalent.
\begin{enumerate}
\item [ (i) ] $n \equiv 0 \mod r$. 
\item [ (ii) ] $L(n+r), L(n + 2r), \ldots, L(n + tr)$ are not all equal.
\end{enumerate}
\end{corollary}

We now describe a uniform way to check whether a particular occurrence of $v$ in $V$ is expected.  

\begin{corollary}
\label{corexp}
Suppose $V$ has an occurrence of $v$ beginning at $i$.  Let $\alpha = 01^{b_1} 0 1^{b_2} 0 \ldots  0 1^{b_s}$ be the subword of $V$ of length $2t|v|$  that begins at $i$.  
\begin{enumerate}
\item [ (a) ] If $b_l \geq |v|$ for some $0 < l \leq s$, then the occurrence of $v$ in $V$ beginning at $i$ is expected iff $l \equiv 0 \mod r$.
\item [ (b) ] If $b_l < |v|$ for all $0 < l \leq s$, then $s > rt$ and the occurrence of $v$ in $V$ beginning at $i$ is expected iff 
$b_r, b_{2r}, \ldots, b_{tr}$ are not all equal.
\end{enumerate}
\end{corollary}

\begin{proof}
Let $n \in \N$ so that $i = i_n$.  Note that for $0 < l < s$, we have $b_l = L(n+l)$.  

Proof of (a):  If $b_l \geq |v|$, then $L(n+l) \geq |v|$.  This implies that $L(n+l) > a_k$ for each $0 < k < r$.  Thus, $n+l \not\equiv k \mod r$ for each $0  < l < r$.  Therefore, $n+l \equiv 0 \mod r$.  Thus, $n \equiv 0 \mod r$ iff $l \equiv 0 \mod r$.  We know that the occurrence of $v$ in $V$ beginning at $i_n$ is expected iff $n \equiv 0 \mod r$.  Therefore, the occurrence of $v$ in $V$ beginning at $i_n = i$ is expected iff $l \equiv 0 \mod r$.

Proof of (b):  Suppose that $b_l < |v|$ for all $0 < l \leq s$.   If $s \leq tr$, then simple counting shows that $|\alpha| \leq t|v| + t (|v| - 1) < 2t|v|$.  Therefore, $s > tr$.

We know that the occurrence of $v$ in $V$ beginning at $i_n$ is expected iff $n \equiv 0 \mod r$.  By Corollary \ref{corexpV}, we know $n \equiv 0 \mod r$ iff $L(n+r), L(n + 2r), \ldots, L(n + tr)$ are not all equal.  But we also know that $L(n + r) = b_r, \ldots, L(n+tr) = b_{tr}$.  Therefore, the occurrence of $v$ in $V$ beginning at $i$ is expected iff 
$b_r, b_{2r}, \ldots, b_{tr}$ are not all equal.
\end{proof}

The crucial fact about Corollary \ref{corexp} is that to determine whether an occurrence of $v$ in $V$ beginning at $i$ is expected it is sufficient to know the subword of $V$ of length $2t|v|$ that begins at $i$.

\subsubsection{Expectedness in an element of a rank-1 system}
\label{secexpelement}

For the rest of Section \ref{secexpelement}, fix $(X, \sigma)$, a non-degenerate rank-1 system and fix $v \in \mathcal{F}$.  Let $v = 0 1^{a_1} 0 1^{a_2} \ldots  1^{a_{r-1}}0$.

\begin{definition}
\label{defbuiltx}
We say that $x \in X$ is {\em built from} $v$ if there is a collection $E_v$ of occurrences of $v$ in $x$ so that each 0 in $x$ is part of exactly one element of $E_v$.
\end{definition}

Remarks:
\begin{enumerate}
\item  If the collection $E_v$ witnesses that $x$ is built from $v$, then distinct elements of $E_v$ must actually be disjoint (since $v$ begins with 0).
\item  If $x$ is constantly 1, i.e., $x$ contains no occurrence of 0, then $x$ is built from $v$ vacuously.
\end{enumerate}

There are three types of elements of $X$ that have an occurrence of 0.  If $x \in X$ is has an occurrence of 0, then either:  
\begin{enumerate}
\item [ (a) ] $x$ has a first occurrence of 0;
\item [ (b) ] $x$ has a last occurrence of 0; or
\item [ (c) ] $x$ has an occurrence of 0, but neither a first nor last occurrence of 0.
\end{enumerate} 

By Corollary \ref{cor} no $x \in X$ can satisfy both (1) and (2) above.  For each of these possibilities we have a lemma, which we state here without proof.

\begin{lemma}
\label{lemma1a}
Suppose $x \in X$ has a first occurrence of 0 and that $E_v$ witnesses that $x$ is built from $v$.  The following statements (about $i \in \Z$) are equivalent:
\begin{enumerate}
\item [ (i) ] $x(i) = 0$ and $|\{j < i : x(j) = 0\}|$ is a multiple of $r$.
\item [ (ii) ] some element of $E_v$ begins at position $i$ in $x$.
\end{enumerate}
\end{lemma}

\begin{lemma}
\label{lemma2a}
Suppose $x \in X$ has a last occurrence of 0 and that $E_v$ witnesses that $x$ is built from $v$.  The following statements (about $i \in \Z$) are equivalent.
\begin{enumerate}
\item [ (i) ] $x(i) = 0$ and $|\{j \geq i : x(j) = 0\}|$ is a multiple of $r$.
\item [ (ii) ] Some element of $E_v$ begins at position $i$ in $x$.
\end{enumerate}
\end{lemma}

If $x \in X$ has an occurrence of 0, but neither a first nor a last occurrence of 0, then let $(i_n : n \in \Z)$ enumerate, in an order-preserving way, the positions of the occurrences of 0 in $x$.  Define $L_x : \Z \rightarrow \N$ by $L_x(n) = i_{n} - i_{n-1} -1$.  So, $L_x(n)$ is the number of 1s between the 0 at position $i_{n-1}$ in $x$ and the 0 at position $i_{n}$ in $x$.  Note that since $x $ is not periodic, neither is $L_x$.

\begin{lemma}
\label{lemma3a}
Suppose $x \in X$ has an occurrence of 0, but neither a first nor last occurrence of 0.  Suppose further that $E_v$ witnesses that $x$ is built from $v$.  Then the following statements (about $n \in \Z$) are equivalent:
\begin{enumerate}
\item [ (i) ] Some element of $E_v$ begins at position $i_n$ in $x$.
\item [ (ii) ] The function $L_x$ is not constant on the congruence class of $n$ (mod $r$), but is constant on each other congruence class mod $r$.
\end{enumerate}
\end{lemma}

\begin{proposition}
\label{propexp1}
If $x \in X$ is built from $v$, then there is a unique collection $E_v$ of occurrences of $v$ in $x$ so that each 0 in $x$ is part of exactly one element of $E_v$.
\end{proposition}

\begin{proof}
If $x$ has no occurrence of 0, then only the empty set will work as $E_v$.  That the proposition holds when $x$ has an occurrence of 0 follows immediately from Lemmas \ref{lemma1a},  \ref{lemma2a}, and  \ref{lemma3a}.\end{proof}

\begin{definition}
\label{defexp}
Let $x \in X$ be built from $v$.  An occurrence of $v$ in $x$ is {\em expected} if it belongs to $E_v$ and {\em unexpected} otherwise.
\end{definition}

The next corollary follows immediately from Definition \ref{defbuiltx}, Proposition \ref{propexp1} and Definition \ref{defexp}.
\begin{corollary}
\label{corexppreceq}
Suppose $x \in X$ is built from $v$ and that $u \preceq v$.  Then $x$ is built from $u$ and every expected occurrence of $u$ in $x$ is completely contained in an expected occurrence of $v$. 
\end{corollary}

\subsubsection{Connections between a rank-1 word and elements of its associated system}
\label{secexpconnect}

For the rest of Section \ref{secexpconnect}, fix $V \in \mathcal{R}$ and let $(X, \sigma)$ be the non-degenerate rank-1 system associated to $V$.

\begin{proposition}
\label{propexp2}
If $V$ is built from $v$, then each $x \in X$ is built from $v$.
\end{proposition}

\begin{proof}
Suppose $V$ is built from $v$ and let $x \in X$.  We first define $E_v$, a collection of occurrences of $v$ in $x$.  Suppose $x$ has an occurrence of $v$ beginning at $i$.  Let $\alpha = 01^{b_1} 0 1^{b_2} 0 \ldots  0 1^{b_s}$ be the subword of $x$ of length $2t|v|$  that begins at $i$. 

\begin{enumerate}
\item [ (a) ] If $b_k \geq |v|$ for some $0 < k \leq s$, then the occurrence of $v$ in $V$ beginning at $i$ belongs to $E_v$ iff $k \equiv 0 \mod r$.
\item [ (b) ] If $b_k < |v|$ for all $0 < k \leq s$, then $s > rt$ and the occurrence of $v$ in $V$ beginning at $i$ belongs to $E_v$ iff $b_r, b_{2r}, \ldots, b_{tr}$ are not all equal.
\end{enumerate}

That each occurrence of 0 is $x$ is a part of exactly one element of $E_v$ follows from Corollary \ref{corexp}.  Indeed, suppose $x$ has an occurrence of 0 at position $i$.  Let $\beta$ be the subword of $x$ beginning at $i - |v| +1$ and ending at $i + 2t|v| -1$.  Note that if $x$ has an occurrence of $v$ (say it begins at $i-r$) that contains the occurrence of 0 at position $i$, then the subword of $x$ of length $2t|v|$ and beginning at $i-r$ is completely contained in the occurrence of $\beta$ beginning at $i - |v| + 1$.

Let $j \in \N$ be such that $V$ has an occurrence of $\beta$ beginning at $j - |v| + 1$.  Note that $V$ has an occurrence of 0 at position $j$.  Note further that if  $V$ has an occurrence of $v$ (say it begins at $j-r$) that contains the occurrence of 0 at position $j$, then the subword of $V$ of length $2t|v|$ and beginning at $j-r$ is completely contained in the occurrence of $\beta$ beginning at $j - |v| + 1$.
 
It is clear that $x$ has an occurrence of $v$ beginning at $i - r$ that contains the occurrence of 0 at position $i$ iff $V$ has an occurrence of $v$ beginning at $j - r$ that contains the occurrence of 0 at position $j$.  Moreover, by Corollary \ref{corexp}, such an occurrence of $v$ in $x$ belongs to $E_v$ iff the corresponding occurrence of $v$ in $V$ is expected.

Since the occurrence of 0 in $V$ at position $j$ is contained in exactly one expected occurrence of $v$, the occurrence of 0 in $x$ at position $i$ is contained in part of exactly one element of $E_v$.
\end{proof}

Corollary \ref{corclopen} follows immediately from the way that $E_v$ was defined in the previous proof.  First, however, we need a definition.

\begin{definition}
$\empty$
\begin{enumerate}
\item [ (a) ] For any finite word $\alpha$ and $ i \in \Z$, $$U_{\alpha, i} = \{x \in X : \textnormal{$V$ has an occurrence of $\alpha$ beginning at $i$}\}$$
\item [ (b) ] For $v$ such that $V$ is built from $v$, $$E_{v, i} = \{x \in X : \textnormal{$V$ has an expected occurrence of $v$ beginning at $i$}\}$$
\end{enumerate}
\end{definition}

\begin{corollary}
\label{corclopen}
If $V$ is built from $v$ and $i \in \Z$, then $E_{v, i}$ is a clopen subset of $X$.
\end{corollary} 

We also have the following.

\begin{proposition}
\label{propexpbasis}
If $U \subseteq X$ is non-empty and open, then there is some $v$ from which $V$ is fundamentally built and some $j \in \Z$ so that $E_{v, j} \subseteq U$.
\end{proposition}

\begin{proof}  Let $U \subseteq X$ be non-empty and open.  Let $\alpha$ be a finite word and $i \in \Z$ so that $U_{\alpha , i} \subseteq U$ is non-empty.  Since $U_{\alpha, i}$ is non-empty, there is some $k \in \N$ so that $V$ has an occurrence of $\alpha$ beginning at $k$.  Choose $v$ to be such that $V$ is fundamentally built from $v$ and so that $|v| \geq k + |\alpha|$.  Note that the occurrence of $v$ in $V$ beginning at $0$ completely contains the occurrence of $\alpha$ in $V$ beginning at $k$.  Thus, if any $x \in X$ has an occurrence of $v$ beginning at $l$, then $x$ has an occurrence of $\alpha$ beginning at $l + k$.  Let $j = i - k$, and notice that $E_{v, j} \subseteq U_{\alpha, i} \subseteq U$.    
\end{proof}

\begin{proposition}
\label{propexp3}
If $x \in X$ is not periodic but is built from $v$, then $v$ is an initial segment of $V$.
\end{proposition}

\begin{proof}
Suppose $x \in X$ is not periodic but is built from $v$.  Choose $w$ so that $V$ is built from $w$ and so that $|v| \leq |w|$.  By Proposition \ref{propexp2}, $x$ is built from $w$.  To show that $v$ is an initial segment of $V$, it suffices to show that there is some $i \in \Z$ so that $x$ has an occurrence of $v$ beginning at $i$ and also an occurrence of $w$ beginning at $i$.  Let $r$ be the number of occurrences of 0 in $v$ and $s$ be the number of occurrences of 0 in $w$.

Suppose $x$ has a first occurrence of 0; say the first occurrence of 0 in $x$ is at position $i$.  By Lemma \ref{lemma1a}, $x$ has an occurrence of $v$ beginning at $i$ and also an occurrence of $w$ beginning at $i$.  

Suppose that $x$ has a last occurrence of 0.  By Corollary \ref{cor}, we know that $x$ has infinitely many occurrences of 0.  Let $i$ be such that $x(i) = 0$ and $|\{j \geq i : x(j) = 0\}| = rs$.  By Lemma \ref{lemma2a}, $x$ has an occurrence of $v$ beginning at $i$ and also an occurrence of $w$ beginning at $i$.  

Now suppose that $x$ has an occurrence of 0, but neither a first nor last occurrence of 0.  We know that $x$ is built from $v$.  By Lemma \ref{lemma3a}, we can choose $k \in \Z$ so that $L_x$ is not constant on the congruence class of $k$ (mod $r$), but is constant on each other congruence class mod $r$.  Since we also know that $x$ is built from $w$, we can similarly choose $l$ so that $L_x$ is not constant on the congruence class of $l$ mod $s$, but is constant on each other congruence class mod $s$.

We claim that there exists $n \in \Z$ so that $n \equiv k \mod r$ and $n \equiv l \mod s$.  Suppose, toward a contradiction, that there is no such $n$.  Then for every $m \in \Z$, either $m$ is not congruent to $k$ mod $r$ or $m$ is not congruent to $l$ mod $s$.  In either case $L_x(m) = L_x(m + rs)$.  This implies that $L_x$ is periodic, which implies that $x$ is periodic.  This is a contradiction.

We now know there is some $n \in \Z$ so that $n \equiv k \mod r$ and $n \equiv l \mod s$.  By Lemma \ref{lemma3a}, some element of $E_v$ begins at position $i_n$ and some element of $E_w$ begins at position $i_n$.  Thus, $x$ has an occurrence of $v$ beginning at $i_n$ and also an occurrence of $w$ beginning at $i_n$.  
\end{proof}

\begin{proposition}
\label{proponetoone}
Let $V$ and $W$ be distinct elements of $ \mathcal{R}$ and let $(X, \sigma)$ and $(Y, \sigma)$ be the rank-1 systems associated to $V$ and $W$, respectively.  Then $X$ and $Y$ do not share any non-periodic points and thus $X \neq Y$.
\end{proposition}

\begin{proof}
Suppose $x \in X \cap Y$ is not periodic.  If $V$ is built from $v$, then by Proposition \ref{propexp2}, $x$ is built from $v$ and by Proposition \ref{propexp3}, $v$ is an initial segment of $W$.  Since $V$ is built from arbitrarily long words, $V$ and $W$ agree on arbitrarily long initial segments and thus are equal.  This is a contradiction.
\end{proof}

\section{Isomorphisms between non-degenerate rank-1 systems}
\label{RS}

\subsection{Stable isomorphisms and replacement schemes}
\label{RSrest}

\begin{definition}
Let $(X, \sigma)$ and $(Y, \nu)$ be non-degenerate rank-1 systems.  A {\em topological isomorphism} (or {\em isomorphism}) from $(X, \sigma)$ to $(Y, \sigma)$ is a homeomorphism $\phi : X \rightarrow Y$ that commutes with $\sigma$.
\end{definition}

We want to understand when two non-degenerate rank-1 systems are isomorphic.  More generally, we want to understand all isomorphisms between non-degenerate rank-1 systems.  

\begin{definition}
An isomorphism $\phi : X \rightarrow Y$ between non-degenerate rank-1 systems is {\em stable} if there exist $v,w \in \mathcal{F}$ such that for all $x \in X$:
\begin{enumerate}
\item  $x$ is built from $v$ and $\phi (x)$ is built from $w$; and
\item  $x$ has an expected occurrence of $v$ beginning at $k$ iff $\phi (x)$ has an expected occurrence of $w$ beginning at $k$.
\end{enumerate}
In this case we also say the pair $(v,w)$ is a {\em replacement scheme} for $\phi$.
\end{definition}

An isomorphism $\phi$ is stable iff $\phi ^{-1}$ is stable, and $(v,w)$ is a replacement scheme for $\phi$ iff $(w,v)$ is a replacement scheme for $\phi^{-1}$.

In Theorem \ref{thmreplacement} below, we will show that every isomorphism between non-degenerate rank-1 systems is nearly stable.  First, however, we will describe conditions on $V$ and $W$ that are necessary and sufficient for the existence of a stable isomorphism between $(X, \sigma)$ and $(Y, \sigma)$. 

\begin{definition}
Let $V, W \in \mathcal{R}$ and $v, w \in \mathcal{F}$.  The pair $(v,w)$ is a {\em replacement scheme} for $V$ and $W$ if:
\begin{enumerate}
\item  $V$ is built from $v$ and $W$ is built from $w$; and
\item  $V$ has an expected occurrence of $v$ beginning at $k$ iff $W$ has an expected occurrence of $w$ beginning at $k$.
\end{enumerate}
\end{definition}

\begin{proposition}
\label{propreplacement}
Let $V, W \in \mathcal{R}$ and $v, w \in \mathcal{F}$.  The pair $(v,w)$ is a replacement scheme for $V$ and $W$ iff $(v, w)$ is a replacement scheme for some stable isomorphism $\phi : X \rightarrow Y$.
\end{proposition}

\begin{proof}
We first outline the proof of the forward implication.  If  $(v,w)$ is a replacement scheme for $V$ and $W$, then we can obtain $W$ from $V$ by replacing every expected occurrence of $v$ by $w$ and then adding or deleting occurrences of 0 as necessary.  For example, if the length of $v$ minus the length of $w$ is $l \geq 0$, then $W$ can be obtained from $V$ by replacing every expected occurrence of $v$ by $w1^l$.  There is a unique function $\phi : X \rightarrow \{0,1\}^\Z$ so that for each $x \in X$, $\phi(x)$ can be obtained from $x$ in the same way that $W$ can be obtained from $V$ (in our example, by replacing every expected occurrence of $v$ by $w1^l$).  It is then straightforward to check that $\phi$ is a homeomorphism from $X$ to $Y$.  It is obvious from the definition of $\phi$ that $\phi$ commutes with $\sigma$ and that $(v, w)$ is a replacement scheme for $\phi$.

We now prove the reverse implication.  Suppose $(v, w)$ is a replacement scheme for an isomorphism $\phi : X \rightarrow Y$.  Let $(v_m : m \in \N)$ and $(w_n : n \in \N)$ be the canonical generating sequences for $V$ and $W$, respectively.

We claim that there is some $x \in X$ so that for all $m \in \N$, $x$ has an expected occurrence of $v_m$ beginning at 0.  For each $m \in \N$ choose some $x_m \in X$ that has an expected occurrence of $v_m$ beginning at 0.  Note that this implies that $x_m$ has an expected occurrence of $x_n$ beginning at 0, for all $n  \leq m$.  Since $X$ is compact, the sequence $(x_m : m \in \mathbb{N})$ has a convergent subsequence; say $x \in X$ is the limit of a convergent subsequence.  It follows from Corollary \ref{corclopen} that for all $m \in \N$, $x$ has an expected occurrence of $v_m$ beginning at 0.

Now consider $\phi (x)$, and let $U$ be the infinite word in $\phi(x)$ that begins at 0.  To show that $(v,w)$ is a replacement scheme for $V$ and $W$, it suffices to show that $U = W$.  

We first claim that $U$ is rank-1.  It is easy to see that for each $v_m$ such that $v \prec v_m$, there is a corresponding word $v^\prime_m$ from which $U$ is built.  For example, if the length of $v$ minus the length of $w$ is $l \geq 0$ and $v_m = v 1^{a_1} v 1^{a_2} \ldots 1^{a_{r-1}} v$, then $U$ is built from the word $$v^\prime_m = w 1^{a_1 +l} w 1^{a_2+ l} \ldots 1^{a_{r-1} + l} w.$$  Since there are infinitely many $v_m$ such that $v \prec v_m$, $U$ is rank-1.

We next claim that $U$ contains every finite subword of $W$.  Indeed, note that since $x$ has an occurrence of $V$ beginning at 0, $x$ has expected occurrences of $v$ beginning at arbitrarily large $k$.  Since $w$ contains an occurrence of 0, there are occurrences of 0 that occur in $\phi (x)$ at arbitrarily large $k$.  This implies that there are expected occurrences of each $w_n$ that occur in $\phi (x)$ at arbitrarily large $k$.  Thus $U$ contains an occurrence of each $w_n$ and thus contains every finite subword of $W$.

Let $(Z, \sigma)$ be the non-degenerate rank-1 system corresponding to $U$.  Since $U$ contains every finite subword of $W$, we know that $Y \subseteq Z$.  Since a rank-1 system is non-degenerate iff it is uncountable (see Proposition \ref{prop1}), and $Y$ is non-degenerate, it must be the case that $Z$ is non-degenerate.  Now from Proposition \ref{proponetoone} it follows that $W = U$.
\end{proof}

Our next major objective is to show that every isomorphism between non-degenerate rank-1 systems is nearly stable.

\begin{theorem}
\label{thmreplacement}
If $\phi : X \rightarrow Y$ is an isomorphism between non-degenerate rank-1 systems, then for some $q \in \Z$, the isomorphism $\phi \circ \sigma^q : X \rightarrow Y$ is stable.
\end{theorem}

As no minimal rank-1 system can be isomorphic to a non-minimal rank-1 system, we can prove the theorem in two cases:  when both $(X, \sigma)$ and $(Y, \sigma)$ are minimal; and when both $(X, \sigma)$ and $(Y, \sigma)$ are non-minimal.  We will do this in Sections \ref{RSminimal} and \ref{RSnonminimal}, respectively.

\begin{corollary}
If $(X, \sigma)$ is a non-degenerate rank-1 system and $\tau$ is an autohomeomorphism of $X$ that commutes with $\sigma$, then for some $i \in \mathbb{Z}$, $\tau = \sigma ^i$.
\end{corollary}

\begin{proof}
The map $\tau$ is an isomorphism between $(X, \sigma)$ and $(X, \sigma)$.  By Theorem \ref{thmreplacement} there is some $q \in \mathbb{Z}$ so that $\tau \circ \sigma^q$ is a stable isomorphism between $(X, \sigma)$ and $(X, \sigma)$.  Let $(v,w)$ be a replacement scheme for $\tau \circ \sigma^q$.  By Proposition \ref{propreplacement} $(v,w)$ is a replacement scheme for $V$ and $V$.  This implies that $v = w$, which implies that $\tau \circ \sigma^q$ is the identity map.  Thus, $\tau = \sigma^{-q}$.
\end{proof}

Before proceeding with the proof of Theorem \ref{thmreplacement}, we state one more important corollary, which gives a nice characterization of when $V,W \in \mathcal{R}$ are isomorphic.  This characterization will be used in Section \ref{seccomplexity} to determine the complexity of the isomorphism relation on $\mathcal{R}$ and in Section \ref{secinverse} to characterize when a non-degenerate rank-1 system is isomorphic to its inverse.

\begin{corollary}
\label{correplacement}
Let $V, W \in \mathcal{R}$.  Then $V$ is isomorphic to $W$ (i.e., there exists an isomorphism $\phi : X \rightarrow Y$) iff there is a replacement scheme $(v,w)$ for $V$ and $W$.
\end{corollary}

\begin{proof}
If $(v,w)$ is a replacement scheme for $V$ and $W$, then by Proposition \ref{propreplacement}, there is an isomorphism $\phi : X \rightarrow Y$.  On the other hand, if $\phi : X \rightarrow Y$ is an isomorphism, then by Theorem \ref{thmreplacement} there is some $q \in \Z$ so that $\phi \circ \sigma^q$ is stable.  Since the isomorphism $\phi \circ \sigma^q : X \rightarrow Y$ is stable, there is some $(v,w)$ that is a replacement scheme for $\phi \circ \sigma^q$.  By Proposition \ref{propreplacement}, this $(v,w)$ is also a replacement scheme for $V$ and $W$.
\end{proof}

\subsection{Proof of Theorem \ref{thmreplacement} in the minimal case}
\label{RSminimal}

\subsubsection{The setup}
\label{seccase1setup}
Let $(X, \sigma)$ and $(Y, \sigma)$ be minimal non-degenerate rank-1 systems and let $\phi : X \rightarrow Y$ be an isomorphism.  Let $V$ and $W$ be the rank-1 words that give rise to $(X, \sigma)$ and $(Y, \sigma)$, respectively.  Let $X_{\max}$ be the largest number of consecutive 1s that occurs in $V$ (this is the same as the largest number of 1s that occurs in an element of $X$).  Let $Y_{\max}$ be the largest number of consecutive 1s that occurs in $W$.

We will use basic techniques to choose  $v_1, w_2, v_3 \in \mathcal{F}$ and $r, s \in \Z$ satisfying all of the following.
\begin{enumerate}
\item  $v_1, v_3 \in A_V$ and $w_2 \in A_W$.
\item  $v_1 \preceq v_3$.
\item  $|v_1| \geq X_{\max}$ and $|v_1| \geq Y_{\max}$.
\item  For all $x \in X$ and all $k \in \Z$:  
\begin{enumerate}
\item  If $x$ has an expected occurrence of $v_3$ beginning at $k$, then $\phi (x)$ has an expected occurrence of $w_2$ beginning at $k + r$.
\item  If $\phi (x)$ has an expected occurrence of $w_2$ beginning at $k$, then $x$ has an expected occurrence of $v_1$ beginning at $k+s$.
\end{enumerate}
\end{enumerate}

Choose $v_1 \in A_V$ so that $|v_1| > X_{\max}$ and $|v_1| > Y_{\max}$.  Consider the non-empty open set $\phi(E_{v_1, 0}) \subseteq Y$.  By Proposition \ref{propexpbasis}, there is some $w_2 \in A_W$ and some $j \in \Z$ so that $E_{w_2, j} \subseteq \phi(E_{v_1, 0})$.  Note that if $\phi (x)$ has an occurrence of $w_2$ beginning at $j$, then $x$ has an expected occurrence of $v_1$ beginning at 0.  Since $\phi$ commutes with $\sigma$, if any $\phi (x) \in Y$ has an occurrence of $w_2$ beginning at any $k \in \Z$, then $x$ has an expected occurrence of $v_1$ beginning at $k - j$.  Let $s = -j$ and note that if any $\phi (x) \in Y$ has an occurrence of $w_2$ beginning at any $k \in \Z$, then $x$ has an expected occurrence of $v_1$ beginning at $k + s$. 

Now consider the non-empty open set $\phi^{-1}(E_{w_2, 0}) \subseteq X$.  By Proposition \ref{propexpbasis}, there is some $v^\prime_3 \in A_V$ and some $i \in \Z$ so that $E_{v^\prime_3, i} \subseteq \phi(E_{w_2, 0})$.  Now let $v_3 = v^\prime_3 \vee v_1$.  Note that $v_1 \preceq v_3$.  Also, since $v^\prime_3 \preceq v_3$, it follows from Corollary \ref{corexppreceq} that $E_{v_3, i} \subseteq E_{v^\prime_3, i}$.  Thus, $E_{v_3, i} \subseteq \phi(E_{w_2, 0})$.

If $x$ has an occurrence of $v_3$ beginning at $i$, then $\phi(x)$ has an expected occurrence of $w_2$ beginning at 0.  Since $\phi$ commutes with $\sigma$, if any $x \in X$ has an occurrence of $v_3$ beginning at any $k \in \Z$, then $\phi(x)$ has an expected occurrence of $w_2$ beginning at $k - i$.  Let $r = -i$ and note that if any $x \in X$ has an occurrence of $v_3$ beginning at any $k \in \Z$, then $\phi(x)$ has an expected occurrence of $w_2$ beginning at $k + r$. 

The following diagram illustrates the information we have so far. The arrows indicate forcing of occurrences of words.

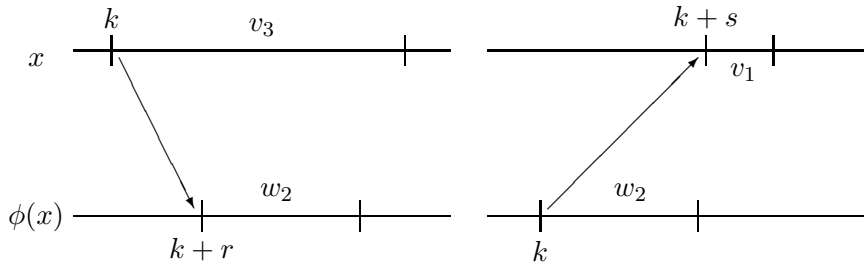
\begin{figure}[h]
\begin{center}
\setlength{\unitlength}{.1cm}
\begin{picture}(120,35)(0,0)
\put(5,6){\makebox(0,0)[b]{$\phi(x)$}}
\put(5,28){\makebox(0,0)[b]{$x$}}
\put(10,8){\line(1,0){50}}
\put(65,8){\line(1,0){50}}
\put(10,30){\line(1,0){50}}
\put(65,30){\line(1,0){50}}

\put(15,28){\line(0,1){4}}
\put(15,33){\makebox(0,0)[b]{$k$}}
\put(16,29){\vector(1,-2){10}}
\put(27,6){\line(0,1){4}}
\put(27,2){\makebox(0,0)[b]{$k+r$}}
\put(54,28){\line(0,1){4}}
\put(35,32){\makebox(0,0)[b]{$v_3$}}
\put(37,10){\makebox(0,0)[b]{$w_2$}}
\put(48,6){\line(0,1){4}}

\put(72,6){\line(0,1){4}}
\put(73,9){\vector(1,1){20}}
\put(93,6){\line(0,1){4}}
\put(94,28){\line(0,1){4}}
\put(103,28){\line(0,1){4}}
\put(94,33){\makebox(0,0)[b]{$k+s$}}
\put(72,2){\makebox(0,0)[b]{$k$}}
\put(99,26){\makebox(0,0)[b]{$v_1$}}
\put(84,10){\makebox(0,0)[b]{$w_2$}}

\end{picture}
\caption{\label{Case1setupfig} Forced occurrences of words $w_2$ and $v_1$.}
\end{center}
\end{figure}

\subsubsection{What we need}  To prove the theorem we need to find finite words $v$ and $w$ and an integer $q$ so that for all $x \in X$ and all $k \in \Z$:
\begin {enumerate}
\item $x$ is built from $v$ and $\phi (x)$ is built from $w$; and
\item $x$ has an expected occurrence of $v$ beginning at $k$ iff $\phi (x)$ has an expected occurrence of $w$ beginning at $k + q$.
\end{enumerate}

The desired $v$ will be the word $ v_3$.  Since $v \in A_V$, Proposition \ref{propexp2} implies that each $x \in X$ is built from $v$.

\subsubsection{First steps} 
Consider any $\phi (x)  \in Y$.  It consists of expected occurrences of $w_2$ interspersed with 1s.  Each expected occurrence of $w_2$ in $\phi (x)$ forces an expected occurrence of $v_1$ in $x$ and, by Corollary \ref{corexppreceq}, that expected occurrence of $v_1$ in $x$ is contained in an expected occurrence of $v_3$.  In this way we associate, to each expected occurrence of $w_2$ in $\phi (x)$, an expected occurrence of $v_3$ in $x$.  Notice that if an expected occurrence of $w_2$ in $\phi (x)$ beginning at $j$ is associated to an expected occurrence of $v_3$ in $x$ beginning at $k$, then $k \leq j + s < k + |v_3|$ or, equivalently, $0 \leq (j-k) + s < |v_3|$.  

We can fix an expected occurrence of $v_3$ in any $x \in X$ (say it begins at $k$) and consider the collection of expected occurrences of $w_2$ in $\phi (x)$ associated to it.  In Proposition \ref{proprepcase1}, we show that the structure of that collection is, in some sense, independent of $x$ and $k$.

\begin{lemma}
\label{lemma11}
$0 \leq r+s < |v_3|$.
\end{lemma}

\begin{proof}
We first prove a claim.

{\bf  Claim:}  Suppose $x$ has an expected occurrence of $v_3$ beginning at $k$, which is followed by $1^a$ and then by another expected occurrence of $v_3$.  Suppose also that the expected occurrence of $v_3$ that contains the expected occurrence of $v_1$ beginning at $k + r +s$ is followed by $1^b$ and then by another expected occurrence of $v_3$.  Then $a = b$.

Here is the proof. We know that $x$ has an expected occurrence of $v_1$ beginning at $k + r + s$.  Thus, there is also an expected occurrence of $v_1$ beginning at $k + r + s + |v_3| + b$.  But we also know that there is an expected occurrence of $v_1$ beginning at $k + |v_3| + a + r + s$.  If $a \neq b$, then these two expected occurrences of $v_1$ are distinct, and hence disjoint.  This clearly implies that $|a-b| \geq |v_1|$, which cannot happen because $a$ and $b$ are non-negative and $|v_1| > X_{\max}$.  This proves the claim.

The next claim can be proved with a similar argument.  We state it without proof.

{\bf  Claim:}
Suppose $x$ has an expected occurrence of $v_3$ beginning at $k$, which is preceded by $1^a$ and then by another expected occurrence of $v_3$.  Suppose also that the expected occurrence of $v_3$ that contains the expected occurrence of $v_1$ beginning at $k + r +s$ is preceded by $1^b$ and then by another expected occurrence of $v_3$.  Then $a = b$.

Now suppose $0 > r + s$ or $r+s \geq |v_3|$.  Choose $x \in X$ and $k \in \Z$ so that $x$ has an expected occurrence of $v_3$ beginning at $k$.  Then $x$ has an expected occurrence of $v_1$ beginning at $k + r + s$.  This expected occurrence of $v_1$ is contained in an expected occurrence of $v_3$ in $x$; say it begins at $k^\prime$.  Note that $k \neq k^\prime$.  Applying the claims repeatedly (infinitely many times) shows that $x$ is periodic, with period dividing $|k - k^\prime|$.  This contradicts Corollary \ref{cor1}.
\end{proof}

The following figure illustrates the information we have so far. 

\begin{figure}[h]
\begin{center}
\setlength{\unitlength}{.1cm}
\begin{picture}(100,40)(0,0)
\put(10,8){\line(1,0){80}}
\put(10,30){\line(1,0){80}}
\put(5,6){\makebox(0,0)[b]{$\phi(x)$}}
\put(5,28){\makebox(0,0)[b]{$x$}}
\put(20,34){\makebox(0,0)[b]{$k$}}
\put(32,3){\makebox(0,0)[b]{$k+r$}}
\put(44,35){\makebox(0,0)[b]{$v_3$}}
\put(41,10){\makebox(0,0)[b]{$w_2$}}
\put(60,25){\makebox(0,0)[b]{$k+r+s$}}
\put(61,32){\makebox(0,0)[b]{$v_1$}}

\put(21,29){\vector(1,-2){10}}
\put(33,9){\vector(1,1){20}}

\put(20,27){\line(0,1){6}}
\put(85,27){\line(0,1){6}}
\put(32,6){\line(0,1){4}}
\put(48,6){\line(0,1){4}}
\put(54,29){\line(0,1){2}}
\put(68,29){\line(0,1){2}}

\end{picture}
\caption{\label{lemma11fig} Forced occurrences of words $w_2$ and $v_1$ in the minimal case.}
\end{center}
\end{figure}
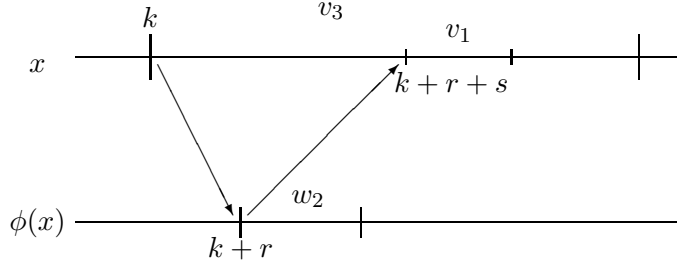

\subsubsection{The key proposition}

\begin{proposition}
\label{proprepcase1}
Suppose $r^\prime$ satisfies $0 \leq r^\prime + s < |v_3|$ and is such that for some $x \in X$ and $k \in \Z$, $x$ has an expected occurrence of $v_3$ beginning at $k$ and $\phi (x)$ has an expected occurrence of $w_2$ beginning at $k + r^\prime$.  Then for every $x \in X$ and $k \in \Z$, if $x$ has an expected occurrence of $v_3$ beginning at $k$, then $\phi (x)$ has an expected occurrence of $w_2$ beginning at $k + r^\prime$.
\end{proposition}

\begin{proof}
Suppose, towards a contradiction, that there is some $r^\prime$ that satisfies the hypothesis of the proposition, but not the conclusion of the proposition.  Either $r > r^\prime$ or $r< r^\prime$.  The arguments that leads to a contradiction in these two cases are almost identical, and we give only one of them.  Suppose $r < r^\prime$.

Let $r_0$ be the largest natural number less than $r^\prime$ that satisfies both the hypothesis and the conclusion of the proposition.  Note that $r_0 \geq r$.  Let $r_1$ be the smallest natural number greater than $r_0$ that satisfies the hypothesis of the proposition but not the conclusion of the proposition.  Note that $r_0 \leq r^\prime$.

Choose $ x_0, k_0, x_1, k_1$ so that all of the following hold.
\begin{enumerate}
\item  $x_0$ has an expected occurrence of $v_3$ beginning at $k_0$.
\item  $\phi (x_0)$ has expected occurrences of $w_2$ beginning at $k_0 + r_0$ and $k_0 + r_1$.
\item  $x_1$ has an expected occurrence of $v_3$ beginning at $k_1$.
\item  $\phi (x_1)$ has an expected occurrence of $w_2$ beginning at $k_1 + r_0$, but not at $k_1 + r_1$.
\end{enumerate}

Let $a$ be such that the expected occurrence of $w_2$ in $\phi (x_0)$ beginning at $k_0+r_0$ is followed by $1^a$ and then by another expected occurrence of $w_2$.

Let $b$ be such that the expected occurrence of $w_2$ in $\phi (x_1)$ beginning at $k_1$ is followed by $1^b$ and then by another expected occurrence of $w_2$.

We claim that $a = r_1 - r_0 - |w_2|$, i.e., that the first expected occurrence of $w_2$ in $\phi (x_0)$ after the expected occurrence of $w_2$ beginning at $k_0 + r_0$ is the one beginning at $k_0 + r_1$.  If not, then there is an expected occurrence of $w_2$ between those beginning at $k_0 + r_0$ and $k_0 + r_1$; say it begins at $k + \hat{r}$, with $r_0 < \hat{r} < r_1$.  If whenever $x \in X$ has an occurrence of $v_3$ at $k$, $\phi (x)$ has an expected occurrence of $w_2$ beginning at $k + \hat{r}$, then we have contradiction with the maximality of $r_0$.  If not, then we have a contradiction with the minimality of $r_1$.  Therefore, $a = r_1 - r_0 - |w_2|$. 

Note that $b \neq a$, for otherwise $\phi(x_1)$ would have an expected occurrence of $w_2$ beginning at $k_1 + r_1$. The following figure illustrates the situation.

\begin{figure}[h]
\begin{center}
\setlength{\unitlength}{.1cm}
\begin{picture}(120,80)(0,0)

\put(0,45){
\put(10,8){\line(1,0){110}}
\put(10,30){\line(1,0){110}}
\put(4,6){\makebox(0,0)[b]{$\phi(x_0)$}}
\put(4,28){\makebox(0,0)[b]{$x_0$}}
\put(20,34){\makebox(0,0)[b]{$k_0$}}
\put(32,3){\makebox(0,0)[b]{$k_0+r_0$}}
\put(55,3){\makebox(0,0)[b]{$k_0+r_1$}}
\put(64,35){\makebox(0,0)[b]{$v_3$}}
\put(41,10){\makebox(0,0)[b]{$w_2$}}
\put(90,25){\makebox(0,0)[b]{$k_0+r_1+s$}}
\put(91,32){\makebox(0,0)[b]{$v_1$}}
\put(21,29){\vector(1,-2){10}}
\put(55,9){\vector(4,3){27}}
\put(20,27){\line(0,1){6}}
\put(115,27){\line(0,1){6}}
\put(32,6){\line(0,1){4}}
\put(48,6){\line(0,1){4}}
\put(83,29){\line(0,1){2}}
\put(97,29){\line(0,1){2}}
\put(52,10){\makebox(0,0)[b]{$1^a$}}
\put(64,10){\makebox(0,0)[b]{$w_2$}}
\put(54,6){\line(0,1){4}}
\put(70,6){\line(0,1){4}}
}

\multiput(83,69)(0,-2){19}{\line(0,-1){1}}

\put(0,0){
\put(10,8){\line(1,0){110}}
\put(10,30){\line(1,0){110}}
\put(4,6){\makebox(0,0)[b]{$\phi(x_1)$}}
\put(4,28){\makebox(0,0)[b]{$x_1$}}
\put(20,34){\makebox(0,0)[b]{$k_1$}}
\put(32,3){\makebox(0,0)[b]{$k_1+r_0$}}
\put(64,35){\makebox(0,0)[b]{$v_3$}}
\put(41,10){\makebox(0,0)[b]{$w_2$}}
\put(58,10){\makebox(0,0)[b]{$1^b$}}
\put(76,10){\makebox(0,0)[b]{$w_2$}}
\put(76,25){\makebox(0,0)[b]{$k_1+r_1+s$}}
\put(102,32){\makebox(0,0)[b]{$v_1$}}
\put(21,29){\vector(1,-2){10}}
\put(67,9){\vector(4,3){27}}
\put(20,27){\line(0,1){6}}
\put(115,27){\line(0,1){6}}
\put(32,6){\line(0,1){4}}
\put(48,6){\line(0,1){4}}
\put(66,6){\line(0,1){4}}
\put(82,6){\line(0,1){4}}
\put(95,29){\line(0,1){2}}
\put(109,29){\line(0,1){2}}
\put(83,29){\line(0,1){2}}
}

\end{picture}
\caption{\label{proprepcase1fig} The proof of Proposition~\ref{proprepcase1}.}
\end{center}
\end{figure}
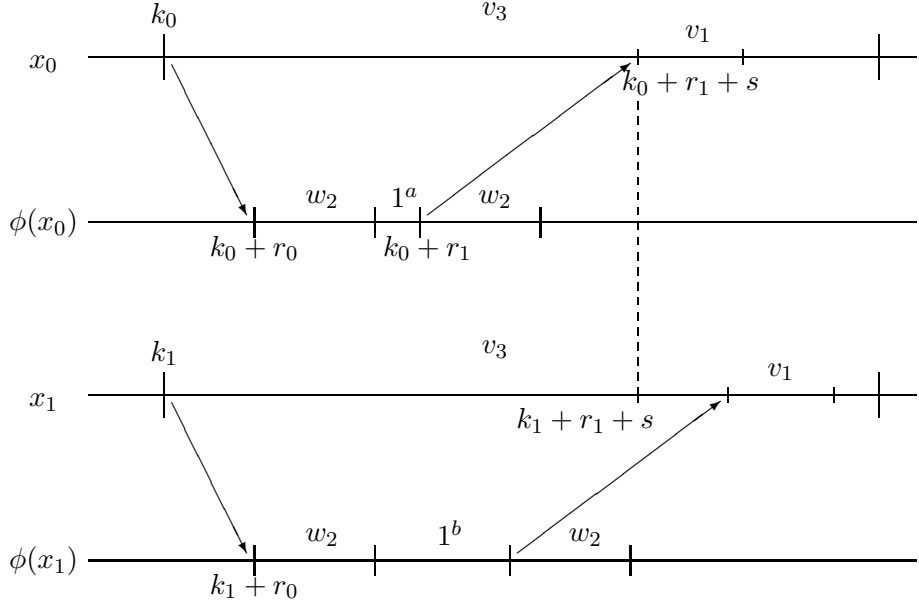

We now claim that $|b-a| \geq |v_1|$.  First, notice that since $\phi(x_0)$ has an occurrence of $w_2$ beginning at $k_0 + r_1$, there is an expected occurrence of $v_1$ in $x_0$ beginning at $k_0 + r_1 + s$.  As $0 \leq r_1  + s < |v_3|$, this expected occurrence of $v_1$ intersects, and thus is contained in, the expected occurrence of $v_3$ beginning at $k_0$ in $x_0$.  This clearly implies that $x_1$ has an expected occurrence of $v_1$ beginning at $k_1 + r_1 + s = k_1 + r_0 + |w_2| + a + s$.  But since $\phi (x_1)$ has an occurrence of $w_2$ beginning at $k_1 + r_0 + |w_2| + b$, there is also an expected occurrence of $v_1$ in $x_1$ beginning at  $k_1 + r_0 + |w_2| + b + s$.  Since $a\neq b$, these two expected occurrences of $v_1$ are distinct, and hence they are disjoint.  This clearly implies that $|b - a| \geq |v_1|$.

Since $a$ and $b$ are both non-negative integers, either $a \geq |v_1|$ or $b \geq |v_1|$.  This contradicts the assumption that $|v_1|> Y_{\max}$.
\end{proof}

\subsubsection{Proof of Theorem \ref{thmreplacement} in the minimal case}

 Choose any $x \in X$ and $k \in \Z$ such that $x$ has an expected occurrence of $v_3$ beginning at $k$.  Let $w$ be the smallest subword of $\phi (x)$ containing every expected occurrence of $w_2$ beginning at a position $k + r^\prime$, for some $0 \leq r^\prime  + s< |v_3|$.  In other words, $w$ is the smallest subword of $\phi (x)$ containing every expected occurrence of $w_2$ that is associated to the expected occurrence of $v_3$ in $x$ beginning at $k$.    
Let $q$ be such that this occurrence of $w$ in $\phi (x)$ begins at $k + q$.  Note if an expected occurrence of $w_2$ in $\phi (x)$ is not associated to the expected occurrence of $v_3$ beginning at $k$ in $x$, then it is disjoint from the occurrence of $w$ at $k + q$.

It follows from Proposition \ref{proprepcase1} that $w$ and $q$ are independent of the choice of $x$ and $k$.  Thus, if $x^\prime \in X$ has an expected occurrence of $v_3$ beginning at $k^\prime \in \Z$, then $\phi (x^\prime)$ has an occurrence of $w$ beginning at $k^\prime + q$.  Moreover, an expected occurrence of $w_2$ in $\phi (x)$ is contained in the occurrence of $w$ beginning at $k^\prime + q$ if it is associated to the expected occurrence of $v_3$ beginning at $k$ in $x$; otherwise it is disjoint from the occurrence of $w$ at $k^\prime + q$.

We have now defined $w$ and $q$ and stated the relevant facts.  To prove the theorem we must show that for all $x \in X$ and all $k \in \Z$:
\begin{enumerate}
\item  $\phi (x)$ is built from $w$; and 
\item  $x$ has an expected occurrence of $v_3$ beginning at $k$ iff $\phi (x)$ has an expected occurrence of $w$ beginning at $k+ q$.
\end{enumerate}

To show that $\phi (x)$ is built from $w$ we need to define a collection of expected occurrences of $w$.  We say an occurrence of $w$ in $\phi (x)$ beginning at $j$ is {\em expected} iff $x$ has an expected occurrence of $v_3$ beginning at $j - q$.  Recall that an expected occurrence of $w_2$ in $\phi (x)$ is contained in an expected occurrence of $w$ beginning at $j$ if it is associated to the expected occurrence of $v_3$ beginning at $j - q$ in $x$; otherwise it does not intersect the expected occurrence of $w$ beginning at $j-q$ in $x$.  Thus every occurrence of 0 is contained in exactly one expected occurrence of $w$, and $\phi (x)$ is built from $w$.  

Condition (2) above immediately follows the definition of expected occurrence of $w$.  This completes the proof of the Theorem \ref{thmreplacement} in the minimal case.

\subsection{Proof of Theorem \ref{thmreplacement} in the non-minimal case}
\label{RSnonminimal}

The proof of Theorem \ref{thmreplacement} is more intricate in the non-minimal case than in the minimal case, though there are significant parts of the argument that are essentially identical.  We will first state and prove Lemma \ref{lemmauniquefirst} and then begin the proof of Theorem \ref{thmreplacement}. 

\begin{lemma}  Let $(X, \sigma)$ be a non-minimal rank-1 system and let $k \in \Z$.
\label{lemmauniquefirst}
\begin{enumerate}
\item[(a)]  There is a unique $z \in X$ so that $z$ has a first occurrence of $0$ at position $k$.
\item[(b)]  There is a unique $z \in X$ so that $z$ has a last occurrence of $0$ at position $k$.
\end{enumerate}
\end{lemma}

\begin{proof}
We will prove (b).  The proof of (a) is similar.

Choose $x \in X$ with an occurrence of 0.  We know that for each $n \in \N$, the word $01^n$ is a subword of $V$ and hence, by Proposition \ref{propfirst}, a subword of $x$; say $x$ has an occurrence of $01^n$ beginning at $k_n$.  For each $n \in \N$, $\sigma^{k_n}(x)$ has an occurrence of $01^n$ beginning at 0.  By passing to a subsequence if necessary, we may assume that $(\sigma^{k_n}(x): n \in \N)$ converges to some $z \in X$.  It is clear that $z$ has a last occurrence of 0 at position 0.

Suppose $z$ and $z^\prime$ are distinct elements of $X$ that each have a last occurrence of $0$ beginning at $k$.  Let $i \in \Z$ be as large as possible so that $z(i) \neq z^\prime (i)$.  Let $t = |\{j \geq i : z(j) = 0\}|$.  Let $v$ have more than $t$ occurrences of 0 and be such that $V$ is built from $v$.  Let $v = 0 1^{a_1} 0 1^{a_2} \ldots  1^{a_{r-1}}0$.  It follows from Lemma \ref{lemma1a} both $z$ and $z^\prime$ must have an occurrence of $v$ beginning at $k - |v| + 1$.  The occurrence of $v$ in each of $z$ and $z^\prime$ must contain the position $i$, contradicting the fact that $z(i) \neq z^\prime(i)$.
\end{proof}

We now begin the proof of Theorem \ref{thmreplacement} in the non-minimal case.

Fix $\phi : X \rightarrow Y$, an isomorphism between non-minimal rank-1 systems $(X, \sigma)$ and $(Y, \sigma)$.  Before we move ahead with the setup and proof as we did in the minimal case, we need an important lemma.

\subsubsection{Preliminary setup and some lemmas}

Using the techniques of Section \ref{seccase1setup}, we can choose $v_0, w_0 \in \mathcal{F}$ and $t, t^\prime \in \Z$ satisfying all of the following.
\begin{enumerate}
\item  $v_0 \in V$ and $w_0 \in W$.
\item  For all $x \in X$ and $k \in Z$:
\begin{enumerate}
\item  If $x$ has an expected occurrence of $v_0$ beginning at $k$, then $\phi (x)$ has an occurrence of 0 at position $k+t$.
\item  If $\phi (x)$ has an expected occurrence of $w_0$ beginning at $k$, then $x$ has an occurrence of 0 at position $k+t^\prime$.
\end{enumerate}
\end{enumerate}

We will use these this preliminary setup (i.e., our choices for $v_0$, $w_0$, $t$, and $t^\prime$) to prove some important lemmas.  Then we produce a different setup (the only thing we will keep from this initial setup is the choice of $t$) that we will use for the main argument.

There are three elements of $X$ that we want to have specific names.  Let $z_0$ denote the bi-infinite word that is constantly 1.  Let $z_1$ denote the element of $X$ that has a first occurrence of 0 at position 0.  Let $z_2$ denote the element of $X$ that has a last occurrence of 0 at position 0.

\begin{lemma}  Both of the following hold.
\begin{enumerate}
\item  $\phi (z_1)$ has a first occurrence of 0.
\item  $\phi (z_2)$ has a last occurrence of 0.
\end{enumerate}
\end{lemma}

\begin{proof}
We will prove (b).  The proof of (a) is similar.

Since $z_2$ has an occurrence of 0, it has an expected occurrence of $v_0$.  This forces an occurrence of 0 in $\phi (z_2)$.  If $\phi (z_2)$ had no last occurrence of 0, then there would arbitrarily large $k$ for which $\phi (z_2)$ has an expected occurrence of $w_2$.  For each such $k$, we know that $x$ must have an occurrence of 0 at position $k + t^\prime$.  This contradicts the fact that $z_2$ has a last occurrence of 0.
\end{proof}

Let $f\in \Z$ be such that $\phi (z_1)$ has a first occurrence of 0 at position $f$.  Let $g \in \Z$ be such that $\phi (z_2)$ has a last occurrence of 0 at position $g$.

\begin{lemma}
\label{lemma1s}
There is some $A \in \N$ so that for $a \geq A$, $x \in X$, and $k \in \Z$, $01^a0$ occurs in $x$ beginning at $k$ iff $01^{a + f - g}0$ occurs in $\phi (x)$ beginning at $k + g$.
\end{lemma}

\begin{proof}
We will prove that for sufficiently large $a$, whenever $x\in X$ has an occurrence of $01^a 0$ beginning at $k$, $\phi (x)$ has an occurrence of $01^{a+f -g}0$ beginning at $k + g$.  Similar reasoning shows that for sufficiently large $a$, whenever $\phi (x)$ has an occurrence of $0 1^{a + f - g} 0$ beginning at $k+g$, $x$ has an occurrence of $01^a0$ beginning at $k$ (thus completing the proof). 

Recall that $z_0$ is the bi-infinite word that is constantly 1.  Note that $z_0 \in X \cap Y$ and that $\phi(z_0) = z_0$, since $z_0$ is the unique fixed point of both $(X, \sigma)$ and $(Y, \sigma)$.

Suppose (towards a contradiction) that it is not true that for sufficiently large $a$, whenever $x\in X$ has an occurrence of $01^a 0$ beginning at $k \in \Z$, $\phi (x)$ has an occurrence of $01^{a+f -g}0$ beginning at $k + g$.  Then we can find sequences $\{x_n\} \subseteq X$, $\{k_n\} \subseteq \Z$, $\{a_n\} \subseteq \N$ so that:
\begin{enumerate}
\item  $a_n \rightarrow \infty$;
\item  $x_n$ has an occurrence of $01^{a_n} 0$ beginning at $k_n$; and
\item  $\phi (x_n)$ does not have an occurrence of $01^{a_n + f - g} 0$ beginning at $k_n + g$.
\end{enumerate}

We know that for each $n$, there is an occurrence of $01^{a_n}$ in $\sigma^{k_n} (x_n)$ beginning at $0$.  We may assume (by passing to a subsequence, if necessary) that the sequence $\{\sigma^{k_n} (x_n)\}$ converges.  It must converge to $z_2$, for the element to which it converges must have a last occurrence of 0 at position 0.  Thus, $\phi (\sigma^{k_n} (x_n)) \rightarrow \phi (z_2)$.

We also know that for each $n$, there is an occurrence of $1^{a_n}0$ in $\sigma^{k_n + a_n + 1} (x_n)$ ending at $0$.  We may assume (by passing to a subsequence, if necessary) that the sequence $\{\sigma^{k_n + a_n + 1} (x_n)\}$ converges.  It must converge to $z_1$, for the element to which it converges must have a first occurrence of 0 at position 0.  Thus, $\phi (\sigma^{k_n + a_n + 1} (x_n)) \rightarrow \phi (z_1)$.

Choose $l\in \N$ large enough that:
\begin{enumerate}
\item  $l \geq |w_0|$;
\item  $l \geq -t^\prime - g$ (equivalently, $0 \leq  g + l + t^\prime$); and
\item  $l \geq f + t^\prime$ (equivalently $f - l + t^\prime \leq 0$).
\end{enumerate}

Choose $n \in \N$ large enough that:
\begin{enumerate}
\item  $\phi (\sigma^{k_n} (x_n))$ has an occurrence of $01^{l}$ beginning at $g$; and
\item  $\phi (\sigma^{k_n + a_n + 1} (x_n))$ has an occurrence of $1^l 0$ ending at $f$.
\end{enumerate}

Since $\phi$ and $\sigma$ commute, we have:
\begin{enumerate}
\item  $\phi (x_n)$ has an occurrence of $01^{l}$ beginning at $k_n + g$; and
\item  $\phi (x_n)$ has an occurrence of $1^l 0$ ending at $k_n + a_n + 1 + f$.
\end{enumerate}

We chose $x_n$, $k_n$, and $a_n$, so that $\phi (x_n)$ did not have an occurrence of $01^{a_n + f - g}0$ beginning at $k_n + g$ (and ending at $k_n + a_n + 1 + f$).  This implies that there must be some occurrence of 0 in $\phi (x_n)$ at position $i$, with $k_n + g + l < i \leq k_n + a_n + f -l$.  But this occurrence of 0 in $\phi (x_n)$ at position $i$ must belong to an expected occurrence of $w_0$; say this expected occurrence of $w_0$ begins at $j$.  Since $l \geq |w_0|$ (in particular, since $1^l$ is not a subword of $w_0$), we know that $k_n + g + l < j \leq k_n + a_n + f -l$.  The following figure illustrates the situation.

\begin{figure}[h]
\begin{center}
\setlength{\unitlength}{.1cm}
\begin{picture}(110,40)(0,0)
\put(10,8){\line(1,0){100}}
\put(10,30){\line(1,0){100}}
\put(4,6){\makebox(0,0)[b]{$\phi(x_n)$}}
\put(4,28){\makebox(0,0)[b]{$x_n$}}
\put(20,25){\makebox(0,0)[b]{$k_n$}}
\put(32,3){\makebox(0,0)[b]{$k_n+g$}}
\put(45,32){\makebox(0,0)[b]{$1^{a_n}$}}
\put(40,10){\makebox(0,0)[b]{$1^l$}}
\put(81,10){\makebox(0,0)[b]{$1^l$}}
\put(39,25){\makebox(0,0)[b]{$j+t'$}}
\put(20,32){\makebox(0,0)[b]{$0$}}
\put(65,32){\makebox(0,0)[b]{$0$}}
\put(53,3){\makebox(0,0)[b]{$j$}}
\put(97,3){\makebox(0,0)[b]{$k_n+a_n+f+1$}}
\put(72,25){\makebox(0,0)[b]{$k_n+a_n+1$}}
\put(32,10){\makebox(0,0)[b]{$0$}}
\put(87,10){\makebox(0,0)[b]{$0$}}
\put(53.5,10){\makebox(0,0)[b]{$0$}}
\put(58,9){\makebox(0,0)[b]{$w_0$}}
\put(53,7){\line(0,1){2}}

\put(52,9){\vector(-1,3){6.8}}
\put(20,29){\line(0,1){2}}

\put(32,7){\line(0,1){2}}
\put(45,7){\line(0,1){2}}
\put(62,7){\line(0,1){2}}
\put(74,7){\line(0,1){2}}
\put(87,7){\line(0,1){2}}

\put(44,29){\line(0,1){2}}
\put(65,29){\line(0,1){2}}

\end{picture}
\caption{\label{lemma1sfig} The proof of Lemma~\ref{lemma1s}.}
\end{center}
\end{figure}
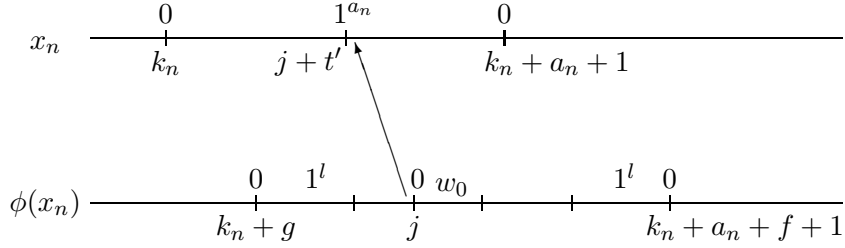

The occurrence of $w_0$ in $\phi (x_n)$ beginning at $j$ forces an occurrence of 0 in $x$ at position $j +t^\prime $.  Clearly:  $$k_n + g + l  + t^\prime  <  j  + t^\prime  \leq k_n + a_n + f - l + t^\prime$$  Since $0 \leq g + l  + t^\prime$ and $f - l + t^\prime \leq 0$, we have: $$k_n < j  + t^\prime \leq k_n + a_n$$  This is a contradiction, since $x_n$ has an occurrence of $0$ at position $j+t^\prime$, and $x_n$ has an occurrence of $01^{a_n} 0$ beginning at $k_n$.
\end{proof}

\subsubsection{Main setup}
We will use basic techniques to choose $v_1,w_2,v_3 \in \mathcal{F}$ and $r, s\in \Z$ satisfying all of the following (recall that $t$ has already been chosen).

\begin{enumerate}
\item  $v_1, v_3 \in A_V$ and $w_2 \in A_W$.
\item  $v_1 \preceq v_3$.
\item  $|v_1| > t - g$ (equivalently, $t < |v_1| + g$) and $ |v_1| > A + |f - g|$.
\item  For all $x \in X$ and all $k \in \Z$:
\begin{enumerate}
\item  If $x$ has an expected occurrence of $v_3$ beginning at $k$, then $\phi (x)$ has an expected occurrence of $w_2$ beginning at $k + r$;
\item  If $\phi (x)$ has an expected occurrence of $w_2$ beginning at $k$, then $x$ has an expected occurrence of $v_1$ beginning at $k+s$; and 
\item  If $x$ has an expected occurrence of $v_1$ beginning at $k$, then $\phi (x)$ has an occurrence of 0 at position $k+t$.
\end{enumerate}
\end{enumerate}

We know that $V$ is built from $v_0$ and that $v_0$ is such that if $x$ has an occurrence of $v_0$ beginning at $k$, then $\phi(x)$ has an occurrence of $0$ beginning at $k+t$.  Choose $v_1 \in B_V$ so that $|v_1| > \max\{|v_0|, t-g, A + |f-g|\}$.  Note that since $|v_1| > |v_0|$ and $v_1 \in B_V$, we have that $v_0 \prec v_1$.    Thus, if $x$ has an expected occurrence of $v_1$ beginning at $k$, then $x$ has an expected occurrence of $v_0$ beginning at $k$, which implies that $\phi (x)$ has an occurrence of 0 at position $k+t$.

We then choose $w_2$ and $s$ and then $v_3$ and $r$ as was done in the minimal case.

\subsubsection{The main argument in the non-minimal case}

We have a slightly different setup here than we did in the minimal case.  Below we state and prove one lemma and give a proof of Proposition \ref{proprepcase1} in the non-minimal case.  The proof of Theorem \ref{thmreplacement} in the non-minimal case proceeds from that proposition in exactly the same way as in the minimal case and we do not repeat that argument here.

\begin{lemma}
\label{lemmar+s}
Both of the following hold.
\begin{enumerate}
\item  $0 \leq r + s < |v_3|$
\item  $0 \leq s + t < |w_2|$
\end{enumerate}
\end{lemma}

\begin{proof}
We will prove (a).  The proof of $(b)$ is similar.

Recall that $z_0 \in X$ has a first occurrence of 0 at position 0.  Thus, it has an expected occurrence of $v_3$ beginning at $0$.  This forces an expected occurrence of $v_1$ in $z_0$ beginning at $r +s$.  Since the first 0 of $z_0$ is at position 0, it must be that $0 \leq r + s$.

Recall that $z_1$ has a last occurrence of 0 at position 0.  Thus, it has an expected occurrence of $v_3$ beginning at $1 - |v_3|$.  This forces an expected occurrence of $v_1$ in $z_1$ beginning at $1 - |v_3| + r + s$.  This clearly cannot begin after position $0$.  Therefore, $1 - |v_3| + r + s \leq 0$.  Thus, $r + s < |v_3|$.
\end{proof}

The following figure illustrates the information we have so far for the non-minimal case.

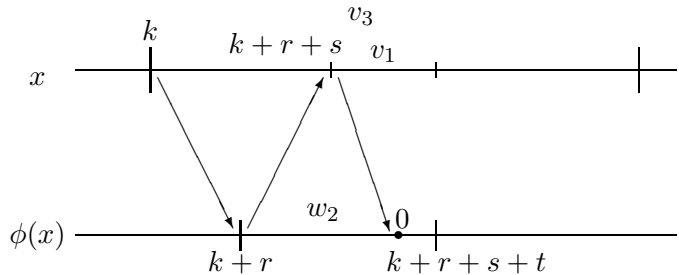
\begin{figure}[h]
\begin{center}
\setlength{\unitlength}{.1cm}
\begin{picture}(100,40)(0,0)
\put(10,8){\line(1,0){80}}
\put(10,30){\line(1,0){80}}
\put(5,6){\makebox(0,0)[b]{$\phi(x)$}}
\put(5,28){\makebox(0,0)[b]{$x$}}
\put(20,34){\makebox(0,0)[b]{$k$}}
\put(32,3){\makebox(0,0)[b]{$k+r$}}
\put(48,36){\makebox(0,0)[b]{$v_3$}}
\put(43,10){\makebox(0,0)[b]{$w_2$}}
\put(38,32){\makebox(0,0)[b]{$k+r+s$}}
\put(51,31){\makebox(0,0)[b]{$v_1$}}
\put(62,3){\makebox(0,0)[b]{$k+r+s+t$}}
\put(21,29){\vector(1,-2){10}}
\put(33,9){\vector(1,2){10}}
\put(45,29){\vector(1,-3){6.8}}
\put(53,8){\circle*{1}}
\put(53.5,9){\makebox(0,0)[b]{$0$}}

\put(20,27){\line(0,1){6}}
\put(85,27){\line(0,1){6}}
\put(32,6){\line(0,1){4}}
\put(58,6){\line(0,1){4}}
\put(44,29){\line(0,1){2}}
\put(58,29){\line(0,1){2}}

\end{picture}
\caption{\label{lemmar+sfig} Forced occurrences in the non-minimal case.}
\end{center}
\end{figure}

We now prove Proposition \ref{proprepcase1} in the non-minimal case

\begin{proof}
Suppose, towards a contradiction, that there is some $r^\prime$ that satisfies the hypothesis of the proposition, but not the conclusion of the proposition.  Either $r > r^\prime$ or $r< r^\prime$.  The arguments that leads to a contradiction in these cases are almost identical, and we give only one of them.  Suppose $r < r^\prime$.

Let $r_0$ be the largest natural number less than $r^\prime$ that satisfies both the hypothesis and the conclusion of the proposition.  Note that $r_0 \geq r$.  Let $r_1$ be the smallest natural number greater than $r_0$ that satisfies the hypothesis of the proposition but not the conclusion of the proposition.  Note that $r_0 \leq r^\prime$.

Choose $ x_0, k_0, x_1, k_1$ so that all of the following hold.
\begin{enumerate}
\item  $x_0$ has an expected occurrence of $v_3$ beginning at $k_0$.
\item  $\phi (x_0)$ has expected occurrences of $w_2$ beginning at $k_0 + r_0$ and $k_0 + r_1$.
\item  $x_1$ has an expected occurrence of $v_3$ beginning at $k_1$.
\item  $\phi (x_1)$ has an expected occurrence of $w_2$ beginning at $k_1 + r_0$, but not at $k_1 + r_1$.
\end{enumerate}

Let $a$ be such that the expected occurrence of $w_2$ in $\phi (x_0)$ beginning at $k_0+r_0$ is followed by $1^a$ and then by another expected occurrence of $w_2$.

We claim that $a = r_1 - r_0 - |w_2|$, i.e., that the first expected occurrence of $w_2$ in $\phi (x_0)$ after the expected occurrence of $w_2$ beginning at $k_0 + r_0$ is the one beginning at $k_0 + r_1$.  If not, then there is an expected occurrence of $w_2$ between those beginning at $k_0 + r_0$ and $k_0 + r_1$; say it begins at $k + \hat{r}$, with $r_0 < \hat{r} < r_1$. If whenever $x \in X$ has an expected occurrence of $v_3$ at $k$, $\phi (x)$ has an expected occurrence of $w_2$ beginning at $k + \hat{r}$, then we have contradiction with the maximality of $r_0$.  Otherwise, we have a contradiction with the minimality of $r_1$.  Therefore, $a = r_1 - r_0 - |w_2|$.

We now claim that there is an occurrence of 0 in $\phi (x_1)$ after the expected occurrence of $w_2$ that begins at $k_1 + r_0$.  Indeed, we know that there is an expected occurrence of $v_1$ in $x_0$ beginning at position $k_0 + r_1 + s$.  As $0 \leq r_1  + s < |v_3|$, this expected occurrence of $v_1$ intersects, and thus is contained in, the expected occurrence of $v_3$ beginning at $k_0$ in $x_0$.  Thus, there also is an expected occurrence of $v_1$ in $x_1$ at position $k_1 + r_1 + s$.  This forces an occurrence of 0 in $\phi (x_1)$ at position $k_1 + r_1 + s + t$.  To show that this occurrence of 0 is after the expected occurrence of $w_2$ that begins at $k_1 + r_0$, it suffices to show that $k_1 + r_0 + |w_2| \leq k_1 + r_1 + s + t$.  Recall that $r_1 - r_0 \geq |w_2|$.  By Lemma \ref{lemmar+s}, we know $0 \leq s + t$.  Thus, $k_1 + r_0 + |w_2| \leq k_1 + r_1 + s + t$.  Therefore, there is an occurrence of 0 in $\phi (x_1)$ after the expected occurrence of $w_2$ that begins at $k_1 + r_0$.

Let $b$ be such that the expected occurrence of $w_2$ in $\phi (x_1)$ beginning at $k_1$ is followed by $1^b$ and then by another expected occurrence of $w_2$.  Note that $b \neq a$, for otherwise $x_1$ would have an expected occurrence of $w_2$ beginning at $k_1 + r_1$.  Also, if $b< a$, there would be a contradiction with the minimality of $r_1$.  Thus $b > a$. Figure~\ref{proprepcase2fig} illustrates the information we have so far.

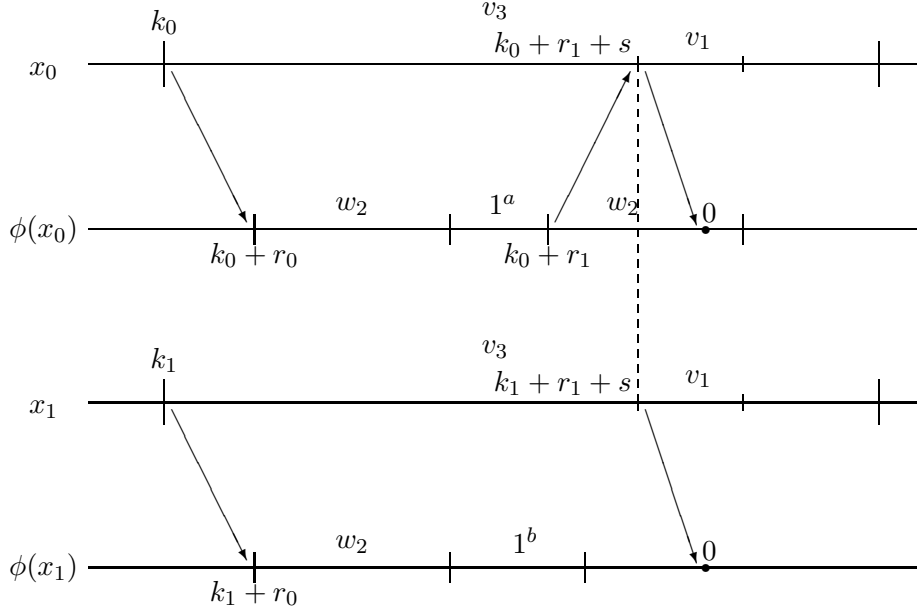
\begin{figure}[h]
\begin{center}
\setlength{\unitlength}{.1cm}
\begin{picture}(120,80)(0,0)

\put(0,45){
\put(10,8){\line(1,0){110}}
\put(10,30){\line(1,0){110}}
\put(4,6){\makebox(0,0)[b]{$\phi(x_0)$}}
\put(4,28){\makebox(0,0)[b]{$x_0$}}
\put(20,34){\makebox(0,0)[b]{$k_0$}}
\put(32,3){\makebox(0,0)[b]{$k_0+r_0$}}
\put(71,3){\makebox(0,0)[b]{$k_0+r_1$}}
\put(64,36){\makebox(0,0)[b]{$v_3$}}
\put(45,10){\makebox(0,0)[b]{$w_2$}}
\put(73,31){\makebox(0,0)[b]{$k_0+r_1+s$}}
\put(91,32){\makebox(0,0)[b]{$v_1$}}
\put(21,29){\vector(1,-2){10}}

\put(20,27){\line(0,1){6}}
\put(115,27){\line(0,1){6}}
\put(32,6){\line(0,1){4}}
\put(58,6){\line(0,1){4}}
\put(83,29){\line(0,1){2}}
\put(97,29){\line(0,1){2}}
\put(65,10){\makebox(0,0)[b]{$1^a$}}
\put(81,10){\makebox(0,0)[b]{$w_2$}}
\put(71,6){\line(0,1){4}}
\put(97,6){\line(0,1){4}}
\put(72,9){\vector(1,2){10}}
\put(84,29){\vector(1,-3){6.8}}
\put(92,8){\circle*{1}}
\put(92.5,9){\makebox(0,0)[b]{$0$}}
}

\multiput(83,73)(0,-2){21}{\line(0,-1){1}}

\put(0,0){
\put(10,8){\line(1,0){110}}
\put(10,30){\line(1,0){110}}
\put(4,6){\makebox(0,0)[b]{$\phi(x_1)$}}
\put(4,28){\makebox(0,0)[b]{$x_1$}}
\put(20,34){\makebox(0,0)[b]{$k_1$}}
\put(32,3){\makebox(0,0)[b]{$k_1+r_0$}}
\put(64,36){\makebox(0,0)[b]{$v_3$}}
\put(45,10){\makebox(0,0)[b]{$w_2$}}
\put(68,10){\makebox(0,0)[b]{$1^b$}}
\put(73,31){\makebox(0,0)[b]{$k_1+r_1+s$}}
\put(91,32){\makebox(0,0)[b]{$v_1$}}
\put(21,29){\vector(1,-2){10}}

\put(20,27){\line(0,1){6}}
\put(115,27){\line(0,1){6}}
\put(32,6){\line(0,1){4}}
\put(58,6){\line(0,1){4}}
\put(76,6){\line(0,1){4}}

\put(97,29){\line(0,1){2}}
\put(83,29){\line(0,1){2}}

\put(84,29){\vector(1,-3){6.8}}
\put(92,8){\circle*{1}}
\put(92.5,9){\makebox(0,0)[b]{$0$}}
}

\end{picture}
\caption{\label{proprepcase2fig} The proof of Proposition~\ref{proprepcase1} in the non-minimal case.}
\end{center}
\end{figure} 

We now claim that $|b-a| \geq |v_1|$.  We know there is an expected occurrence of $v_1$ in $x_0$ beginning at $k_0 + r_1 + s$.  As $0 \leq r_1  + s < |v_3|$, this expected occurrence of $v_1$ intersects, and thus is contained in, the expected occurrence of $v_3$ beginning at $k_0$ in $x_0$.  Thus, there also is an expected occurrence of $v_1$ in $x_1$ at position $k_1 + r_1 + s = k_1 + r_0 + |w_2| + a + s$.  Since $\phi (x_1)$ has an expected occurrence of $w_2$ beginning at $k_1 + r_0 + |w_2| + b$, there is also an expected occurrence of $v_1$ in $x_1$ beginning at  $k_1 + r_0 + |w_2| + b + s$.  Since $a\neq b$, these two expected occurrences of $v_1$ are distinct, and hence they are disjoint.  This implies that $|b - a| \geq |v_1|$.  

Since $b>a$ and $|b-a| \geq |v_1|$, we know that $b \geq |v_1| > A + |f-g|$.  Recall that $\phi(x_1)$ has an occurrence of $01^b0$ beginning at $k_1 + r_0 + |w_2| - 1$.  Therefore, by Lemma \ref{lemma1s}, there is an occurrence of $01^{b - f + g} 0$ in $x_1$ beginning at $k_1 + r_0 + |w_2| - 1 - g$.  

We claim that this occurrence of $01^{b - f + g} 0$ is completely contained in the expected occurrence of $v_3$ beginning at $k_1$.  It suffices to show that the position of the last occurrence of 0 contained in the expected occurrence of $v_3$ beginning at $k_1$ in $x_1$ is after position $k_1 + r_0 + |w_2| - 1 - g$.  Because $x_0$ has an expected occurrence of $v_1$ beginning at $k_0 + r_0 + |w_2| + a + s$ that is completely contained in the expected occurrence of $v_3$ beginning at $k_0$, $x_1$ has an expected occurrence of $v_1$ beginning at $k_1 + r_0 + |w_2| + a + s$ that is completely contained in the expected occurrence of $v_3$ beginning at $k_1$.  Since $v_1$ ends in 0, there is an occurrence of 0 at position $k_1 + r_0 + |w_2| + a + s + |v_1| - 1$ in $x$ that is completely contained in the expected occurrence of $v_3$ beginning at $k_1$.  Thus, it suffices to show: $$k_1 + r_0 + |w_2| - 1 - g < k_1 + r_0 + |w_2| + a + s + |v_1| - 1$$  This is the same as showing: $$0 < a + s  + |v_1| + g$$  We know that $0 \leq a$.  From Lemma \ref{lemmar+s}, we also know that $0 \leq s + t$.  We chose $v_1$ so that $t < |v_1| + g$.  Thus,  $$0 \leq a + s + t < a + s + |v_1| + g$$  

We now know that the occurrence of $01^{b -f + g} 0$ in $x_1$ beginning at $k_1 + r_0 + |w_2| - 1 - d$ is completely contained in the expected occurrence of $v_3$ beginning at $k^\prime$.  This implies that there is an occurrence of $01^{b - f + g} 0$ in $x_0$ beginning at $k_0 + r_0 + |w_2| - 1 - g$.  By Lemma \ref{lemma1s} there is an occurrence of $01^b0$ in $\phi(x_0)$ beginning at $k_0 + r_0 + |w_2| - 1$.  But we know there is an occurrence of $01^a0$ in $\phi(x_0)$ beginning at $k_0 + r_0 + |w_2| - 1$.  Thus, $a = b$, a contradiction.
\end{proof}

This completes the proof of Theorem \ref{thmreplacement} in the non-minimal case.

\subsection{The complexity of the isomorphism relation on $\mathcal{R}$}
\label{seccomplexity}

We want to understand the complexity of the (topological) isomorphism relation on $\mathcal{R}$ as a subset of $\mathcal{R} \times \mathcal{R}$ and also as a Borel equivalence relation. 

By Corollary \ref{correplacement} we know that $V$ and $W$ are isomorphic (i.e., the rank-1 systems associated to $V$ and $W$ are isomorphic) iff there are $v,w \in \mathcal{F}$ so that:
\begin{enumerate}
\item  $V$ is built from $v$ and $W$ is built from $w$; and
\item  $V$ has an expected occurrence of $v$ beginning at $k$ iff $W$ has an expected occurrence of $w$ beginning at $k$.
\end{enumerate}

We claim that given $v, w \in \mathcal{F}$, the conjunction of the two conditions above is closed.  To see this, first note that the statement ``$V$ is built from $v$" is equivalent to ``for all $n \in \N$ there is a finite word that begins with $V \restriction n$ and that is built from $v$," which is a closed condition.  This implies that the first condition above is closed.  Then note that if $V$ is built from $v$, then the truth of ``$V$ has an expected occurrence of $v$ beginning at $k$" depends only on the first $k + |v|$ values of $V$.  This implies that the conjunction of the two conditions is closed.

Since the set $\mathcal{F}$ is countable, we have shown the following.

\begin{proposition}
The isomorphism relation on $\mathcal{R}$ is $F_\sigma$ as a subset of $\mathcal{R} \times \mathcal{R}$.
\end{proposition}

We now want to understand the complexity of the isomorphism relation on $\mathcal{R}$ as a Borel equivalence relation.  We will use without elaboration the most basic results about hyperfinite Borel equivalence relations. These results, together with a discussion of the notion of Borel reducibility, can be found in \cite{DoughertyJacksonKechris}. 

\begin{theorem}
The isomorphism relation on $\mathcal{R}$ is Borel bi-reducible with $E_0$.
\end{theorem}

We will prove the above theorem in two stages.  In Proposition \ref{propreduction} we will show that there is a Borel reduction from $E_0$ to the isomorphism relation on $\mathcal{R}$.  It should be mentioned that the map that we show is a Borel reduction is not new.  It was shown in \cite{Fieldsteel} that if $\alpha, \beta \in 2^\N$ are not $E_0$ related, then the image (under this map) of $\alpha$ and $\beta$ are not measure-theoretically isomorphic.  The map is also discussed in \cite{delJuncoRaheSwanson} and \cite{Rudolph}.   

Then in Proposition \ref{prophyperfinite}, we will show that the topological isomorphism relation on $\mathcal{R}$ is hyperfinite.  By a theorem of Dougherty, Jackson, and Kechris (see \cite{DoughertyJacksonKechris}), this is equivalent to the existence of a Borel reduction from the topological isomorphism relation on $\mathcal{R}$ to $E_0$.

Before stating and proving Propositions \ref{propreduction} and \ref{prophyperfinite}, we give a definition and a lemma.

\begin{definition}
Suppose $v, \tilde{v}, w, \tilde{w} \in \mathcal{F}$ with $\tilde{v} = v 1^{a_1} v 1^{a_2} \ldots 1^{a_{r-1}}v$ and $w =w 1^{b_1}w 1^{b_2} \ldots 1^{b_{r-1}}w$.   We say that {\it $\tilde{w}$ is built from $w$ in the same way that $\tilde{v}$ is built from $v$} if for each $0 < i < r$, $a_i + |v| = b_i + |w|$.
\end{definition}

Remarks:
\begin{enumerate}
\item  This is the natural finite analog of the definition of a replacement scheme for $V$ and $W$ (with $\tilde{v}$ playing the role of a finite $V$ and $\tilde{w}$ playing the role of a finite $W$).
\item  If $(v,w)$ is a replacement scheme for $V$ and $W$ and $v \prec \tilde{v} \in A_V$, then the unique $\tilde{w} \in \mathcal{F}$ such that $\tilde{w}$ is built from $w$ in the same way that $\tilde{v}$ is built from $v$ is such that $\tilde{w} \in A_W$.  Moreover, $(\tilde{v}, \tilde{w})$ is a replacement scheme for $V$ and $W$.
\end{enumerate}

\begin{lemma}
\label{lemmaRS1}
Suppose $(v,w)$ is a replacement scheme for $V$ and $W$ and $v \prec \tilde{v} \in B_V$.  Then there is some $\tilde{w} \in B_W$ so that $(\tilde{v}, \tilde{w})$ is a replacement scheme for $V$ and $W$.
\end{lemma}

\begin{proof}
Choose $\tilde{w} \in A_W$ so that $\tilde{w}$ is built from $w$ in the same way that $\tilde{v}$ is built from $v$.  Then $\tilde{w} \in A_W$ and $(\tilde{v}, \tilde{w})$ is a replacement scheme for $V$ and $W$.  We need further to show that $\tilde{w} \in B_W$.

Suppose that $\tilde{w} \notin B_W$.  Then there exist $u,u^\prime \in A_W$ so that $u \prec \tilde{w} \prec u^\prime$ and $u \preceq_s u^\prime$.  We now have three possibilities:  $w \preceq u$, $u \prec w$, or $w$ and $u$ are incomparable.

Case 1:  Suppose $w \preceq u$.  Let $t$ ($t^\prime$) be such that $t$ ($t^\prime$) is built from $v$ in the same way that $u$ ($u^\prime$) is built from $w$.  It is straightforward to check that $t \prec \tilde{v}\prec t^\prime$ and $t \preceq_s t^\prime$.  Thus, $\tilde{v} \notin B_V$, which is a contradiction.

Case 2:  Suppose $u \prec v$.  Then, by the second part of Lemma \ref{lemmasimply}, we know that $w \preceq_s u^\prime$.  We also know that $w \prec \tilde{w} \prec u^\prime$.  Now let, as in the previous case,  $t^\prime$ be such that $t^\prime$ is built from $v$ in the same way that $u^\prime$ is built from $w$.  It is straightforward to check that $v \prec \tilde{v} \prec t^\prime$ and $v \preceq_s t^\prime$.  Thus, $\tilde{v} \notin B_V$, which is a contradiction.

Case 3:  Now suppose that $u$ and $w$ are incomparable.  By Proposition \ref{propincomparable} we know that $(u \wedge w) \prec u \prec (u \vee w)$ and $(u \wedge w) \preceq_s (u \vee w)$.  But we also know that $u \vee w \prec u^\prime$ (since $\tilde{w} \prec u^\prime$ and $\tilde{w}$ is built from each of $u$ and $w$) and that $u \preceq_s u^\prime$.  Thus, by the first part of Lemma \ref{lemmasimply} we know that $u \wedge w \preceq_s u^\prime$.  Then by the second part of Lemma \ref{lemmasimply} we know that $w \preceq_s u^\prime$.  Now let, as in the previous two cases,  $t^\prime$ be such that $t^\prime$ is built from $v$ in the same way that $u^\prime$ is built from $w$.  It is straightforward to check that $v \prec \tilde{v} \prec t^\prime$ and $v \preceq_s t^\prime$.  Thus, $\tilde{v} \notin B_V$, which is a contradiction.
\end{proof}

\subsubsection{The reduction from $E_0$ to $\mathcal{R}$}
\label{nonsmooth}

\begin{proposition}
\label{propreduction}
There is a Borel reduction from $E_0$ to the topological isomorphism relation on $\mathcal{R}$.
\end{proposition}

\begin{proof}  We construct a Borel function from $2^\N$ to $\mathcal{R}$ so that $\alpha \in 2^\N$ and  $\beta \in 2^\N$ are $E_0$-related (i.e., they agree on all but finitely many coordinates) iff their images are topologically isomorphic.

For $\alpha \in 2^\N$, we produce $V$ as follows.  We first let $v_0 = 0$ and then inductively define: 
\[
v_{m + 1}=
     \begin{cases}
	v_m 1 v_m v_m, & \textnormal{if  }  \alpha(m) = 0;\\
	v_m v_m 1 v_m, & \textnormal{if  }  \alpha(m) = 1.
     \end{cases}
\]
Note that each $v_m$ is a proper initial segment of $v_{m+1}$.  Let $V$ be the limit of the $v_m$.  It is clear that $V$ is built from each $v_m$ and thus is rank-1.    We will show below that, in fact, $A_V = \{v_m : m \in \N\} = B_V$.  We will use the fact that there is no $w \in V$ so that $v_m \prec w \prec v_{m+1}$ (this follows from the way that $v_{m+1}$ is built from $v_m$).

We first claim that $A_V = \{v_m : m \in \N\}$.  We already showed above that $A_V \subseteq \{v_m : m \in \N\}$.  Suppose, towards a contradiction, that $V$ is built from $v$ but that $v \notin \{v_m : m \in \N\}$.  Let $m$ be as large as possible so that $v_m \prec v$ (note that such an $m$ exists, since $0 = v_0 \prec v$).  By assumption, $v \neq v_{m+1}$ and $v_{m+1} \nprec v$.  We also know that $v \nprec v_{m+1}$, for otherwise we would have $v_m \prec v \prec v_{m+1}$.  Thus, $v$ and $v_{m+1}$ are incomparable.  By Proposition \ref{propincomparable}, we have that $(v \wedge v_{m+1}) \prec v_{m+1} \prec (v \vee v_{m+1})$ and $(v \wedge v_{m+1}) \preceq_s (v \vee v_{m+1})$.  Thus, by Lemma \ref{lemmasimply} we have that $(v \wedge v_{m+1}) \preceq_s v_{m+1}$.  Since $v_m \preceq v \wedge v_{m+1} \prec v_{m+1}$, it must be that $v_m =    v \wedge v_{m+1}$.  Thus $v_m \preceq_s v_{m+1}$, which contradicts the way that $v_{m+1}$ is built from $v_m$.  

We now have shown that $A_V = \{v_m : m \in \N\}$.  That $B_V = \{v_m : m \in \N\}$ follows from the fact that no $v_{m+1}$ is built simply from $v_m$.  Since $B_V$ is infinite, we know that $V$ is non-degenerate, and thus, in $\mathcal{R}$ (see Proposition \ref{propBinfinite}).

It is clear that this map is continuous and, therefore, Borel.  

Let $\alpha, \beta \in 2 ^\N$.  Let $\alpha$ produce the sequence of words $(v_m: m \in \N)$ and the infinite word $V$.  Let $\beta$ produce the sequence of words $(w_n : n \in \N)$ and the infinite word $W$.  We need to show that $\alpha$ and $\beta$ are $E_0$-related iff $V$ and $W$ are isomorphic (by Corollary \ref{correplacement}, this happens iff there is a replacement scheme for $V$ and $W$).

If $\alpha$ and $\beta$ are $E_0$-related, then there is some $N\in \N$ so that $\alpha (n) = \beta(n)$, for all $n \geq N$.  It is easy to check that this implies that $(v_N, w_N)$ is a replacement scheme form $V$ and $W$.  

Now suppose there is a replacement scheme for $V$ and $W$.  To show that $\alpha (n) = \beta(n)$, for sufficiently large $n \geq N$, it suffices to show that $(v_n, w_n)$ is a replacement scheme for $V$ and $W$, for sufficiently large $n$.  Indeed, if both $(v_n, w_n)$ and $(v_{n+1}, w_{n+1})$ are replacement schemes for $V$ and $W$ then $w_{n+1}$ must be built from $w$ in the same way that $v_{n+1}$ is built from $v_n$; this clearly implies that $\alpha(n) = \beta(n)$.

We first claim that there is some $n \in \N$ such that $(v_n, w_n)$ is a replacement scheme for $V$ and $W$.  We know there is a replacement scheme for $V$ and $W$.  We also know that $A_V = \{v_m : m \in \N\}$ and that $A_W = \{w_n : n \in \N\}$.  Let $m, n \in \N$ be such that $(v_m, w_n)$ is a replacement scheme for $V$ and $W$.  We need to show that $m = n$.  Suppose, towards a contradiction, that $m \neq n$ and assume, without loss of generality, that $m < n$.  Note that $|v_n| = |w_n|$ and that $|v_m| = |w_m|$.  We know that $V$ has an expected occurrence of $v_m$ beginning at $|v_m|$ or $|v_m| + 1$ (depending on whether $\alpha(m) = 0$ or $\alpha(m) = 1$), but that $W$ does not have an expected occurrence of $w_n$ at either $|v_m|$ or $|v_m| + 1$ (the second expected occurrence of $w_n$ in $W$ begins either at $|w_n|$ or $|w_n| +1$).  This contradicts the fact that $(v_m, w_n)$ is a replacement scheme for $V$ and $W$.

We now claim that if $(v_n, w_n)$ is a replacement scheme for $V$ and $W$, then $(v_{n+1}, w_{n+1})$ is a replacement scheme for $V$ and $W$.  Suppose $(v_n, w_n)$ is a replacement scheme for $V$ and $W$.  We know that $v_n \prec v_{n+1} \in B_V$ and thus, by Proposition \ref{lemmaRS1}, we know that for some $n^\prime$ so that $(v_{n+1}, w_{n^\prime})$.  The same argument used in the preceding paragraph shows that $n^\prime = n+1$.  Thus, $(v_{n+1}, w_{n+1})$ is a replacement scheme for $V$ and $W$.
\end{proof}

\subsubsection{Hyperfiniteness}  
\label{hyperfinite}

\begin{proposition}
\label{prophyperfinite}
The topological isomorphism relation on $\mathcal{R}$ is hyper\-finite.
\end{proposition}

\begin{proof}
We will show that the isomorphism relation on $\mathcal{R}$ is a countable union of finite equivalence relations.  In order to define the equivalence relations, we need a norm on replacement schemes.  Suppose $(v,w)$ is a replacement scheme for $V$ and $W$ and that $$V = v 1^{a_1} v 1^{a_2} \ldots  \hspace{.5in}  \textnormal{and} \hspace{.5in} W = w 1^{b_1} w 1^{b_2} \ldots$$  Note that since $(v,w)$ is a replacement scheme, $a_i + |v| = b_i + |w|$, for each $i \geq 1$.  We now define $||(v,w)|| =|v| + \min \{a_i : i \geq 1\} = |w| + \min\{b_i : i \geq 1\}$.  Note that if $(u,v)$ is a replacement scheme for $U$ and $V$ and $(v,w)$ is a replacement scheme for $V$ and $W$, then $(u,w)$ is a replacement scheme for $U$ and $W$ and $||(u,v)|| = ||(v,w)|| = ||(u,w)||$.

For $k \geq 1$ and $V, W \in \mathcal{R}$, we say $V \sim_k W$ iff there is a replacement scheme $(v,w)$ for $V$ and $W$ such that $v \in B_V$, $w \in B_W$, and $||(v,w)|| \leq k$.  

We first show that each $\sim_k$ is an equivalence relation.  It is clear that $\sim_k$ is reflexive and symmetric.  We need to show each $\sim_k$ is transitive.  Suppose $U,V,W \in \mathcal{R}$ with $U \sim_k V$ and $V \sim_k W$.  Let $(u,\tilde{v})$ be a replacement scheme witnessing that $U \sim_k V$ and $(v, w)$ be a replacement scheme witnessing that $V \sim_k W$.  Without loss of generality, assume that $|v| \leq |\tilde{v}|$.  Since $v, \tilde{v} \in B_V$, we know that $v \preceq \tilde{v}$ (see Proposition \ref{propcompatible}).  Now, by Lemma \ref{lemmaRS1}, there is $\tilde{w} \in B_W$ so that $(\tilde{v}, \tilde{w})$ is a replacement scheme.  Now $(u,\tilde{w})$ is a replacement scheme for $U$ and $W$ and that $||(u,\tilde{w})|| = ||(u,\tilde{v})||$.  Thus $U \sim_k W$.

It is easy to see that each $\sim_k$ equivalence class is finite (thus, that each $\sim_k$ is a finite equivalence relation).  Indeed, let $V \in \mathcal{R}$ and notice that if $(v,w)$ witnesses that $V \sim_k W$, then $v, w \in \mathcal{F}$ with $|v|, |w| \leq k$.  As there are only finitely many pairs $(v,w)$ such that $v, w \in \mathcal{F}$ and $|v|, |w| \leq k$, there are only finitely many $W \in \mathcal{R}$ such that $V \sim_k W$.

Finally, we now show that the isomorphism relation on $\mathcal{R}$ is the union of the equivalence relations $\sim_k$.  It is clear that if $V \sim_k W$, then $V$ and $W$ are isomorphic.  On the other hand, if $V, W \in \mathcal{R}$ are isomorphic, then by Corollary \ref{correplacement} there is some replacement scheme $(v,w)$ for $V$ and $W$.  Choose $\tilde{v} \in B_V$ so that $v \preceq \tilde{v}$.  By Lemma \ref{lemmaRS1}, there is $w \in B_W$ so that $(\tilde{v},\tilde{w})$ is a replacement scheme for $V$ and $W$.  Then $V \sim_{||(\tilde{v},\tilde{w})||} W$.

We have shown that each $\sim_k$ is a finite equivalence relations and that their union is the isomorphism relation on $\mathcal{R}$.  Thus, the isomorphism relation on $\mathcal{R}$ is hyperfinite.
\end{proof}

\subsection{The inverse problem for non-degenerate rank-1 systems}
\label{secinverse}
We want to know when a non-degenerate rank-1 system $(X, \sigma)$ is (topologically) isomorphic to its inverse, $(X, \sigma^{-1})$.

\begin{definition}  For various objects $o$, we define the {\em reverse} of $o$, denoted by $\overline{o}$, as follows.
\begin{enumerate}
\item  For a finite word $\alpha = (\alpha_0, \alpha_1, \ldots, \alpha_n)$, let $\overline{x} = (\alpha_n, \ldots, \alpha_1, \alpha_0)$.  
\item  For an bi-infinite word $x \in \{0,1\}^\Z$, let $\overline{x}$ be the unique bi-infinite word such that $\overline{x}(k) = x(-k)$ for all $k \in \Z$. 
\item  For a set of bi-infinite words $X$, let $\overline{X} = \{\overline{x}: x \in X\}$.  
\item  For $V \in \mathcal{R}$ with canonical generating sequence $(v_n : n \in \N)$, let $\overline{V}$ denote the unique infinite word such that $\overline{v_n}$ is an initial segment of $\overline{V}$ for each $n \in \N$.  
\end{enumerate}
\end{definition}

Remarks:
\begin{enumerate}
\item  If $V \in \mathcal{R}$ has canonical generating sequence $(v_n : n \in \N)$, then $\overline{V} \in \mathcal{R}$ and has canonical generating sequence $(\overline{v_n}: n \in \N)$.
\item  If $X$ is the non-degenerate rank-1 system associated $V$, then $\overline{X}$ the non-degenerate rank-1 system associated to $\overline{V}$.  Moreover,  $(\overline{X}, \sigma)$ is topologically isomorphic to $(X, \sigma ^{-1})$.
\end{enumerate}

\begin{definition}
Suppose $w = v 1^{a_1} v 1^{a_2} \ldots 1^{a_{r-1} }v$.  Then we say that $w$ is built {\em symmetrically} from $v$ if for all $0 < i < r$, $a_i = a_{r-i}$.
\end{definition}

We remark that for $v,w \in \mathcal{F}$ with $v \preceq w$, the following are equivalent:
\begin{enumerate}
\item [ (i) ] $w$ is symmetrically built from $v$.
\item [ (ii) ] $w$ is built from $v$ in the same way that $\overline{w}$ is built from $\overline{v}$.
\end{enumerate}

\begin{proposition}  Let $V \in \mathcal{R}$ with canonical generating sequence $(v_n : n \in \N)$.  The following are equivalent:
\begin{enumerate}
\item [ (i) ] $(X, \sigma)$ is topologically isomorphic to $(X, \sigma^{-1})$.
\item [ (ii) ] There is some $N \in \N$ so that for all $n \geq N$, $v_{n+1}$ is built symmetrically from $v_n$.
\end{enumerate}
\end{proposition}

\begin{proof}  

Suppose $N \in \N$ is such for all $n \geq N$, $v_{n+1}$ is built symmetrically from $v_n$.    It is straightforward to check that $(v_N, \overline{v_N})$ is a replacement scheme for $V$ and $\overline{V}$.  Thus, $(X, \sigma)$ and $(\overline{X}, \sigma)$ are topologically isomorphic, which implies that $(X, \sigma)$ and $(X, \sigma^{-1})$ are topologically isomorphic.

Now suppose $(X, \sigma)$ is isomorphic to $(X, \sigma^{-1})$.  Then $(X, \sigma)$ and $(\overline{X}, \sigma)$ are topologically isomorphic.  By Corollary \ref{correplacement}, there is a replacement scheme $(v,w)$ for $V$ and $\overline{V}$.  Choose $M \in \N$ so that $|v| < |v_M|$.  Then for any $m \geq M$, $v \prec v_m$ (because $|v| < |v_m|$ and $v_m$ is comparable with every element of $A_V$).  By Lemma \ref{lemmaRS1} there is some $n$ so that $(v_m, \overline{v_n})$ is a replacement scheme for $V$ and $\overline{V}$.  It is straightforward to check that this can only happen when $m = n$.  Let $v_{n+1} = v_n 1^{a_1} v_n 1^{a_2} \ldots 1^{a_{r-2}}v_n 1^{a_{r-1}} v_n$.  By the definition of the reverse of a finite word, we know that $$\overline{v_{n+1}} = \overline{v_n} 1^{a_{r-1}}  \overline{v_n}  1^{a_{r-2}} \ldots 1^{a_{2}} \overline{v_n}  1^{a_{1}}  \overline{v_n}. $$
But we also know that $(v_n, \overline{v_n})$ is a replacement scheme for $V$ and $\overline{V}$.  Thus, $$\overline{v_{n+1}} = \overline{v_n} 1^{a_{1}}  \overline{v_n}  1^{a_{2}} \ldots 1^{a_{r-2}} \overline{v_n}  1^{a_{r-1}}  \overline{v_n}. $$
Thus, $\overline{v_{n+1}}$ is built symmetrically from $\overline{v_n}$, which implies that $v_{n+1}$ is built symmetrically from $v_n$.
\end{proof}

\medskip
\noindent{\sc Acknowledgment.} We thank Ben Miller for discussions on an earlier version of this paper.

\end{document}